\documentclass[final,leqno,onefignum,onetabnum]{siamltex1213}
\usepackage{algorithm}
\usepackage{overpic}
\usepackage{eepic}
\usepackage{graphicx}
\usepackage{subfigure}
\usepackage{algorithmic}
\usepackage{array}
\usepackage{bm}
\usepackage{amsfonts}
\usepackage{amsmath}
\usepackage{amssymb}
\usepackage{mathrsfs}
\DeclareMathOperator*{\argmin}{arg\,min}
\newtheorem{thm}{Theorem}[section]
\newtheorem{remark}[thm]{Remark}

\title{Low-rank Approximation of Tensors via Sparse Optimization}

\author{Xiaofei Wang\thanks{Key Laboratory for Applied Statistics of MOE, School of Mathematics and Statistics, Northeast Normal University, Renmin Street 5268, Changchun, China
(\email{wangxf341@nenu.edu.cn})}\and
Carmeliza Navasca\thanks{Department of Mathematics, University of Alabama at Birmingham, 1300 University Boulevard, Birmingham AL, USA
(\email{cnavasca@uab.edu})}}

\begin{document}
\maketitle
\slugger{sisc}{sisc}{xx}{x}{x--x}

\begin{abstract}
The goal of this paper is to find a low-rank approximation for a given tensor.
Specifically, we give a computable strategy on calculating the rank of a given tensor,
based on approximating the solution to an
NP-hard problem. In this paper, we formulate a sparse optimization problem via an $l_1$-regularization to find a low-rank approximation of tensors.
To solve this sparse optimization problem, we propose a rescaling algorithm of the proximal alternating minimization
and study the theoretical convergence of this algorithm.
Furthermore, we discuss the probabilistic consistency of the sparsity result
and suggest a way to choose the regularization parameter for practical computation.
In the simulation experiments, the performance of our algorithm supports that our method provides
an efficient estimate on the number of rank-one tensor components in a given tensor.
Moreover, this algorithm is also applied to surveillance videos for low-rank approximation.
\end{abstract}



\pagestyle{myheadings}
\thispagestyle{plain}
\markboth{LOW-RANK APPROXIMATION OF TENSORS}{XIAOFEI WANG AND CARMELIZA NAVASCA}

\section{Introduction}
We have seen the success of the matrix SVD for several decades. However in the advent of modern and massive datasets, even SVD has its limitation. 
Since tensors have been known to be a natural representation of higher-order and hierarchical dimensional datasets, we focus on the extension of low rank matrix approximation to tensors.  Tensors have received much attention in recent years in the areas of signal processing \cite{comon,sidiropoulos,sorensen}, computer vision \cite{hao,martin,savas,xu}, neuroscience \cite{beckmann,mart}, data science and machine learning \cite{kolda,xu,hao,acardunlavykolda}. Most of these applications rely on decomposing a tensor data into its low-rank form to be able to perform efficient computing as well as to reduce memory requirements. This type of tensor decomposition into a sum of rank-one tensor terms is called the canonical polyadic (CP) decomposition; thus, it is viewed as a generalization of the matrix SVD.  The generalization of matrix SVD to tensors is not unique. Another tensor decomposition is called the Higher-Order SVD \cite{lieventucker,kolda}, which is a product of orthogonal matrices with a dense core tensor. Higher-order SVD is considered another extension of the matrix SVD.

Unlike the matrix case where the low-rank matrix approximation is afforded by truncating away \emph{small} rank-one matrix terms \cite{EY}, discarding negligible rank-one tensor terms does not necessarily provide the best low-rank tensor approximation \cite{kolda2}. Moreover, most low rank tensor algorithms do not provide an estimation on the tensor rank; an a priori tensor rank is often required to find the decomposition. Several theoretical results \cite{kruskal,Landsberg} on tensor rank can help, but they are limited to low-multidimensional and low order tensors so they are inapplicable to tensors in real-life applications.
In fact, for real dataset, tensor rank is important. In a source apportionment data problem \cite{hopke}, the tensor rank of the data provides the number of pollution source profiles to be identified.
In this work, the focus is on finding an estimation of the tensor rank and its rank-one tensor decomposition (CP) of a given tensor. There are several numerical techniques \cite{comon,domanov,kolda,navasca,sidiropoulos} for approximating a $k$th rank tensor into its CP decomposition, but they do not give an approximation of the minimum rank. There are algorithms \cite{comon2,brachat} which give tensor rank, but they are specific to symmetric tensor decomposition over the complex field using algebraic geometry tools.

Our proposed algorithm addresses two difficult problems for the CP decomposition:
(a) one is that finding the rank of tensors is a NP-hard problem \cite{hillar},
and (b) the other is that tensors can be ill-posed \cite{silva} and failed to have their best low-rank approximations.

The tensor rank problem is formulated as an $l_0$ minimization problem; i.e.
\begin{eqnarray} \label{nphard}
\min\limits_{\bm{\alpha}} \|\bm{\alpha}\|_0 \quad\mbox{subject~to~}\   \mathcal{A}=[\bm{\alpha};\mathbf{X,Y,Z}]_R
\end{eqnarray}
where $\sum\limits_{r=1}^{R} \alpha_r\mathbf{x}_r\circ\mathbf{y}_r\circ\mathbf{z}_r=[\bm{\alpha};\mathbf{X,Y,Z}]_R$ represents a sum of the outer products of the vectors $\mathbf{x}_r, \mathbf{y}_r,$ and $\mathbf{z}_r$ for $r=1,\ldots,R$. Here $\|\bm{\alpha}\|_0$ corresponds the number of nonzero coefficients in the sum. However, this problem formulation (\ref{nphard}) is NP-hard. Inspired by the techniques in compressive sensing \cite{candes1,candes2,foucart},
we then consider an $l_1$-regularization formulation for tensor rank,
\begin{eqnarray} \label{l1reg}
\min\limits_{\bm{\alpha}} \|\bm{\alpha}\|_1 \quad\mbox{subject~to~}\   \mathcal{A}=[\bm{\alpha};\mathbf{X,Y,Z}]_R.
\end{eqnarray}
Here we denote $\Vert \bm{\alpha} \Vert_{1}= \sum_{i=1}^{R} \vert \alpha_i \vert$. It is well known in the compressed sensing community that minimizing the $\ell_1$ norm of the vector $\bm{\alpha}$ recovers the sparse solution of the linear system.
In the presence of noise, the constraint, $\mathcal{A}=[\bm{\alpha};\mathbf{X,Y,Z}]_R$ is replaced with $\Vert \mathcal{A}-[\bm{\alpha};\mathbf{X,Y,Z}]_R \Vert_F \leq\varepsilon$ where $\Vert \cdot \Vert_F$ is the Frobenius norm with $\Vert \mathcal{A} \Vert_F=(\sum \mathcal{A}_{ijk}^2)^{\frac{1}{2}}$. Moreover, to achieve a tensor decomposition as well as tensor rank, we minimize over the factor matrices, $\mathbf{X,Y,}$ and $\mathbf{Z}$, thus this minimization problem is considered:
\begin{eqnarray} \label{l1regnoise}
\min\limits_{\mathbf{X,Y,Z},\bm{\alpha}}\frac{1}{2}\|\mathcal{A}-[\bm{\alpha};\mathbf{X,Y,Z}]_R\|_F^2+\lambda\|\bm{\alpha}\|_1.
\end{eqnarray}
The $\ell_1$-regularization achieves a good approximation of tensor rank due to the sparsity structure and its tractability. In addition, the $l_1$-regularization term provides a restriction on the boundedness of the variables thereby ameliorating the ill-posedness of the best low-rank approximation of tensors. For more tractable computing, an alternative multi-block constraint optimization \cite{bolte} is implemented, which is similar to the technique discussed in \cite{xu}. Since (\ref{l1regnoise}) is a minimization of a sum of a smooth term
and a nonsmooth term, we consider the following optimization problem with smooth and non smooth terms:
\begin{eqnarray}\label{updatetech}
\mathbf{x}^{k+1}=\arg\min\limits_{\mathbf{x}} \{f(\mathbf{x}^k)+\langle\mathbf{x}-\mathbf{x}^k,\nabla_{\mathbf{x}} f(\mathbf{x}^k)\rangle+\frac{t}{2}\|\mathbf{x}-\mathbf{x}^k\|^2+g(\mathbf{x})\}
\end{eqnarray}
where $f$ and $g$ are the smooth and nonsmooth functions, respectively. Here $f$ is approximated at a given point $\mathbf{x}^k$.

\subsection{Contributions}
Here we list our contributions in this paper:
\begin{enumerate}
\item
We develop an iterative technique for tensor rank approximation given that the main objective function contains a nonsmooth $l_1$-regularization term. The proximal alternating minimization technique \cite{bolte,xu} has been adapted and rescaled for our tensor rank minimization problem.
\item
We provide some theoretical results on the convergence of our algorithm. We show the objective function satisfies a descent property in Lemma \ref{descent} and a subdifferential lower bound \cite{bolte}. A monotonically decreasing objective function is ensured on the sequence generated by the algorithm.
Furthermore, we point out that the sequence generated by algorithm converges to a critical point of the objective function with indicator functions on the normalization constraint that all the columns of the factor matrices have length one.
\item
For practical implementation, we provide a technique (as well as theoretical results) to find a suitable choice on the regularization parameter directly from the data. The regularization parameter choice has remained a very challenging problem \cite{novati,schwarz,KilmerOLeary,Morozov} in applied inverse problems. Our technique is based on the probabilistic consistency of the sparsity in the classical model found in \cite{wainwright,zhao}:
$$\mathbf{b}=\mathbf{B}\bm{\theta}^*+\bm{\varepsilon},$$
where $\bm{\theta}^*$ is a sparse signal, $\mathbf{B}$ is a design matrix and
$\bm{\varepsilon}$ is a vector of independent subgaussian entries with
mean zero and parameter $\sigma^2$.
We show that to find the true sparsity structure with a high probability,
the regularization parameter relies on two intrinsic parameters $\sigma^2$ and $\gamma$ of models,
where $\sigma^2$ represents the variance of noise,
and $\gamma$ is the incoherence parameter \cite{wainwright} on design matrix $\mathbf{B}$.
The relationship between the regularization and intrinsic parameters actually provides us a suggestion on
how to choose a reasonable regularization parameter for practical computation. To illustrate the performance of this low-rank approximation method,
our experiment consists of four parts.
In the first part, we show the relationship between the regularization parameter and the estimated rank.
In the second part, we estimate the number of rank-one components for given tensors
by adaptively selecting the regularization parameter $\lambda$.
In the third one, we compare our algorithm with a modified alternating least-squares algorithm.
In the last one, we handle the real surveillance video data.
\end{enumerate}

\subsection{Organization}
Our paper is organized as follows. In Section 2, we provide some notations and terminologies
used throughout this paper. In Section 3,
we formulate an $l_1$-regularization optimization to the low-rank approximation of tensors.
In Section 4, we propose an algorithm to solve this $l_1$-regularization optimization
by using a rescaling version of the proximal alternating minimization technique.
In Section 5 we discuss the probabilistic consistency of the sparse optimal solution
and give a suggestion on how to choose the regularization parameter.
The numerical experiments in Section 6 consist of simulated and real datasets.
Finally, our conclusion and future work are given in Section 7.

\section{Notation}
We denote a vector by a bold lower-case letter $\mathbf{a}$.
The bold upper-case letter $\mathbf{A}$ represents a matrix
and the symbol of tensor is a calligraphic letter $\mathcal{A}$.
Throughout this paper, we focus on third-order tensors
$\mathcal{A}=(a_{ijk})\in\mathbb{R}^{I\times J\times K}$ of three indices $1\leq i\leq I,1\leq j\leq J$ and $1\leq k\leq K$,
but all the methods proposed here can be also applied to tensors of arbitrary high order.

A third-order tensor $\mathcal{A}$ has column, row and tube fibers,
which are defined by fixing every index but one and
denoted by $\mathbf{a}_{:jk}$, $\mathbf{a}_{i:k}$ and $\mathbf{a}_{ij:}$ respectively.
Correspondingly, we can obtain three kinds $\mathbf{A}_{(1)},\mathbf{A}_{(2)}$ and $\mathbf{A}_{(3)}$ of matricization
of $\mathcal{A}$ according to respectively arranging the column, row, and tube fibers to be columns of matrices.
We can also consider the vectorization for $\mathcal{A}$ to
obtain a row vector $\mathbf{a}$ such the elements of $\mathcal{A}$ are arranged
according to $k$ varying faster than $j$ and $j$ varying faster than $i$, i.e.,
$\mathbf{a}=(a_{111},\cdots,a_{11K},a_{121},\cdots,a_{12K},\cdots,a_{1J1},\cdots,a_{1JK},\cdots)$.

The outer product $\mathbf{x}\circ\mathbf{y}\circ\mathbf{z}\in\mathbb{R}^{I\times J\times K}$ of three nonzero vectors
$\mathbf{x}, \mathbf{y}$ and $\mathbf{z}$ is a rank-one tensor with elements $x_iy_jz_k$ for all the indices.
A canonical polyadic (CP) decomposition of $\mathcal{A}\in\mathbb{R}^{I\times J\times K}$ expresses $\mathcal{A}$ as a sum of rank-one outer products:
\begin{equation}\label{cpd}
\mathcal{A}=\sum_{r=1}^{R} \mathbf{x}_r\circ\mathbf{y}_r\circ\mathbf{z}_r
\end{equation}
where $\mathbf{x}_r\in\mathbb{R}^I,\mathbf{y}_r\in\mathbb{R}^J,\mathbf{z}_r\in\mathbb{R}^K$ for $1\leq r\leq R$.
Every outer product $\mathbf{x}_r\circ\mathbf{y}_r\circ\mathbf{z}_r$ is called as a rank-one component
and the integer $R$ is the number of rank-one components in tensor $\mathcal{A}$.
The minimal number $R$ such that the decomposition (\ref{cpd}) holds is the rank of tensor $\mathcal{A}$, which is denoted by $\mbox{rank}(\mathcal{A})$.
For any tensor $\mathcal{A}\in\mathbb{R}^{I\times J\times K}$,
$\mbox{rank}(\mathcal{A})$ has an upper bound $\min\{IJ,JK,IK\}$ \cite{kruskal}.

The CP decomposition (\ref{cpd}) can be also written as:
\begin{equation}\label{cpd2}
\mathcal{A}=\sum_{r=1}^{R} \alpha_r\mathbf{x}_r\circ\mathbf{y}_r\circ\mathbf{z}_r
\end{equation}
where $\alpha_r\in\mathbb{R}$ is a rescaling coefficient of rank-one tensor $\mathbf{x}_r\circ\mathbf{y}_r\circ\mathbf{z}_r$ for $r=1,\cdots,R$.
For convenience, we let $\bm{\alpha}=(\alpha_1,\cdots,\alpha_R)\in\mathbb{R}^R$
and $[\bm{\alpha};\mathbf{X},\mathbf{Y},\mathbf{Z}]_R = \sum_{r=1}^{R} \alpha_r\mathbf{x}_r\circ\mathbf{y}_r\circ\mathbf{z}_r$
in (\ref{cpd2}) where $\mathbf{X}=(\mathbf{x}_1,\cdots,\mathbf{x}_R)\in\mathbb{R}^{I\times R},\mathbf{Y}=(\mathbf{y}_1,\cdots,\mathbf{y}_R)\in\mathbb{R}^{J\times R}$ and $\mathbf{Z}=(\mathbf{z}_1,\cdots,\mathbf{z}_R)\in\mathbb{R}^{K\times R}$
are called the factor matrices of tensor $\mathcal{A}$.
We impose a normalization constraint on factor matrices such that each column is normalized to length one \cite{kolda,uschmajew} which is denoted by $\mathbf{N(X,Y,Z)}=1$.
For most alternating optimization algorithms for tensors, \emph{flattening} the tensor (matricization) is necessary to be able to break down the problem into several subproblems. Here we describe a standard approach for a matricizing of a tensor. The Khatri-Rao product of two matrices $\mathbf{X}\in\mathbb{R}^{I\times R}$ and $\mathbf{Y}\in\mathbb{R}^{J\times R}$
is defined as
$$\mathbf{X\odot Y}=(\mathbf{x}_1\otimes\mathbf{y}_1,\cdots,\mathbf{x}_R\otimes\mathbf{y}_R)\in\mathbb{R}^{IJ\times R},$$
where the symbol ``$\mathbf{\otimes}$" denotes the Kronecker product:
$$\mathbf{x\otimes y}=(x_1y_1,\cdots,x_1y_J,\cdots,x_Iy_1,\cdots,x_Iy_J)^T.$$
Using the Khatri-Rao product, the decomposition (\ref{cpd2}) can be written in three different matrix forms of tensor $\mathcal{A}$ \cite{bro}:
\begin{equation}\label{khatri1}
\mathbf{A}_{(1)}=\mathbf{XD}(\mathbf{Z\odot Y})^T, \mathbf{A}_{(2)}=\mathbf{YD}(\mathbf{Z\odot X})^T, \mathbf{A}_{(3)}=\mathbf{ZD}(\mathbf{Y\odot X})^T\\
\end{equation}
where the matrix $\mathbf{D}$ is diagonal with elements of $\bm{\alpha}$.

\section{Sparse optimization for low-rank approximation}
The main goal of this study is to find a tensor of low-rank of the original tensor efficiently and accurately.
We first formulate a tensor rank optimization problem:
\begin{eqnarray*}
\min_{\mathcal{B}} \mbox{rank}(\mathcal{B}) \quad\mbox{subject~to~}\  \Vert \mathcal{A} -\mathcal{B} \Vert^2_F < \epsilon.
\end{eqnarray*}
For any given error $\varepsilon$,
the minimal rank of $\mathcal{B}$ such that $\|\mathcal{A}-\mathcal{B}\|_F^2\leq\varepsilon$
is no larger than $\mbox{rank}(\mathcal{A})$.
The optimal solution $\hat{\mathcal{B}}$ is
a low-rank approximation of $\mathcal{A}$ with error $\varepsilon$.

We represent the tensor $\mathcal{B}$ as
$\sum\limits_{r=1}^{R} \alpha_r\mathbf{x}_r\circ\mathbf{y}_r\circ\mathbf{z}_r=[\bm{\alpha};\mathbf{X,Y,Z}]_R$
where $R$ is a upper bound of the rank of $\mathcal{A}$ and columns of $\mathbf{X,Y,Z}$ satisfy
the normalization constraint $\mathbf{N(X,Y,Z)}=1$. Rescaling the columns of the matrices $\mathbf{X,Y,Z}$ is a standard technique \cite{uschmajew,acardunlavykolda}. It is implemented in practice for canonical polyadic tensor decomposition to prevent the norm of the approximated matrices blowing up to infinity while another factor matrix tend to zero while keeping the residual small.

The tensor rank minimization is equivalent to the following constraint optimization problem with $l_0$-norm:
\begin{equation}\label{problem2}
\min\limits_{\bm{\alpha}}\|\bm{\alpha}\|_0 \quad\mbox{s.t.}\ \|\mathcal{A}-[\bm{\alpha};\mathbf{X,Y,Z}]_R\|_F^2\leq\varepsilon,\mathbf{N(X,Y,Z)}=1
\end{equation}
The problem (\ref{problem2}) is equivalent to that of finding the rank of tensors when $\varepsilon=0$,
whose decision version is NP-hard \cite{hillar}.

To make it more tractable, we turn to an optimization problem with $l_1$-norm:
\begin{equation}\label{problem3}
\min\limits_{\bm{\alpha}}\|\bm{\alpha}\|_1 \quad\mbox{s.t.}\ \|\mathcal{A}-[\bm{\alpha};\mathbf{X,Y,Z}]_R\|_F^2\leq\varepsilon,\mathbf{N(X,Y,Z)}=1
\end{equation}
Furthermore, we then solve:
\begin{equation}\label{problem4}
\min\limits_{\mathbf{X,Y,Z},\bm{\alpha}}\frac{1}{2}\|\mathcal{A}-[\bm{\alpha};\mathbf{X,Y,Z}]_R\|_F^2+\lambda\|\bm{\alpha}\|_1 \quad\mbox{s.t.}\ \mathbf{N(X,Y,Z)}=1,
\end{equation}
an $l_1$-regularization optimization problem in which it includes the factor matrices as primal variables.
These optimization formulations are common in compressed sensing \cite{candestao,donoho,candesrombergtao,candes1,candes2,foucart}.
By introducing the indicator function, we switch the constrained optimization problem (\ref{problem4}) into the following unconstrained form:
\begin{equation}\label{problem5}
\min\limits_{\mathbf{X,Y,Z},\bm{\alpha}} \frac{1}{2}\|\mathcal{A}-[\bm{\alpha};\mathbf{X,Y,Z}]_R\|_F^2+\lambda\|\bm{\alpha}\|_1
+\delta_{S_1}(\mathbf{X})+\delta_{S_2}(\mathbf{Y})+\delta_{S_3}(\mathbf{Z})
\end{equation}
where $S_1=\{\mathbf{X}|\|\mathbf{x}_r\|=1,r=1,\cdots,R\},S_2=\{\mathbf{Y}|\|\mathbf{y}_r\|=1,r=1,\cdots,R\}$ and $S_3=\{\mathbf{Z}|\|\mathbf{z}_r\|=1,r=1,\cdots,R\}$.

\begin{remark}
Here there is no simple manner
to compute the relationship between $\varepsilon$ and $\lambda$ without already
knowing the optimal solutions of formulations (\ref{problem3}) and (\ref{problem4}).
In the matrix versions of Basis Pursuit:
$$\min\limits_{\bm{\theta}} \|\bm{\theta}\|_1, \mbox{s.t.}\ \|\mathbf{b}-\mathbf{B}\bm{\theta}\|\leq \varepsilon$$ and
$$\min\limits_{\bm{\theta}}  \frac{1}{2}\|\mathbf{b}-\mathbf{B}\bm{\theta}\|^2+\lambda\|\bm{\theta}\|_1,$$
it is possible to create a mapping between the two parameters through a Pareto curve to
estimate the relationship from the support of few solutions \cite{berg}.
\end{remark}

Our algorithm is tailored for solving the problem (\ref{problem5}). Let the objective function in (\ref{problem5}) as
$$\Psi(\mathbf{X,Y,Z,}\bm{\alpha}):\mathbb{R}^{I\times R}\times\mathbb{R}^{J\times R}\times\mathbb{R}^{K\times R}\times\mathbb{R}^R\rightarrow\mathbb{R}^+,$$
where
$$
\Psi(\mathbf{X,Y,Z,}\bm{\alpha})=f(\mathbf{X,Y,Z,}\bm{\alpha}) + g(\bm{\alpha})+\delta_{S_1}(\mathbf{X})+\delta_{S_2}(\mathbf{Y})+\delta_{S_3}(\mathbf{Z})
$$
with the approximation term $f(\mathbf{X,Y,Z,}\bm{\alpha})=\frac{1}{2}\|\mathcal{A}-[\bm{\alpha};\mathbf{X,Y,Z}]_R\|_F^2$,
the regularized penalty term $g(\bm{\alpha})=\lambda\|\bm{\alpha}\|_1$
and three indicator functions $\delta_{S_1}(\mathbf{X})$, $\delta_{S_2}(\mathbf{Y})$, $\delta_{S_3}(\mathbf{Z})$.
The function $f(\bullet)$ is a real polynomial function on $(\mathbf{X,Y,Z,}\bm{\alpha})$ and the function $g(\bullet)$ is a
non-differential continuous function on $\bm{\alpha}$.
Since $S_1,S_2,S_3$ are closed sets, indicator functions $\delta_{S_1}(\mathbf{X})$, $\delta_{S_2}(\mathbf{Y})$ and $\delta_{S_3}(\mathbf{Z})$
are proper and lower semicontinuous.
Moreover, since $\delta_{S_1}(\mathbf{X})$, $\delta_{S_2}(\mathbf{Y})$ and $\delta_{S_3}(\mathbf{Z})$
are three semi-algebraic functions, thus the objective function is also a semi-algebraic function.
So it is a Kurdyka-{\L}ojasiewicz (KL) function \cite{bolte}.
For a point $\bm{\omega}=(\mathbf{X,Y,Z,}\bm{\alpha})\in\mathbb{R}^{I\times R}\times\mathbb{R}^{J\times R}\times\mathbb{R}^{K\times R}\times\mathbb{R}^R$,
if its (limiting) subdifferential \cite{bolte}, denoted by $\partial \Psi(\bm{\omega})$, contains $\mathbf{0}$, we call it a critical point of $\Psi(\bullet)$.
The set of critical points of $\Psi(\bullet)$ is denoted by $C_\Psi$.

Due to  ill-posedness \cite{silva,lim} of the best low-rank approximation of tensors, it is known that the problem
of finding a best rank-R approximation for tensors of order 3 or higher has no solution in general. However, after introducing the $l_1$ penalty term $\lambda\|\bm{\alpha}\|_1$ to the low-rank approximation term $f(\bullet)$,
it is always attainable for the minimization of the objective function in (\ref{problem5}).
Thus we have the following theorem to show the existence of the global optimal solution of problem (\ref{problem5}).

\begin{theorem}\label{theorem2}
The global optimal solution of problem (\ref{problem5}) exists.
\end{theorem}
\begin{proof}
For any tensor $\mathcal{A}\in\mathbb{R}^{I\times J\times K}$,
the objective function $\frac{1}{2}\|\mathcal{A}-[\bm{\alpha};\mathbf{X,Y,Z}]_R\|^2_F+\lambda\|\bm{\alpha}\|_1+\delta_{S_1}(\mathbf{X})+\delta_{S_2}(\mathbf{Y})+\delta_{S_3}(\mathbf{Z})$ is
denoted as $\Psi(\mathbf{X,Y,Z,}\bm{\alpha})$.
Notice that all the columns of $\mathbf{X,Y,Z}$ in problem (\ref{problem5}) are constrained to have length one.
We define the $d-$dimensional unit sphere as $\Delta^d=\{\mathbf{v}\in\mathbb{R}^d|\|\mathbf{v}\|_2=1\}$,
and a set $S=\{(\mathbf{X,Y,Z,}\bm{\alpha})\in(\Delta^{I})^R\times(\Delta^{J})^R\times(\Delta^{K})^R\times\mathbb{R}^R\}$.
Since this function $\Psi(\bullet)$ is continuous on $S$, we only need to show that there is a point $s\in S$ such that $\Psi(s)=\inf\{\Psi(x)|x\in S\}$, i.e., the minimization of low-rank approximation with $l_1$ penalty is attainable.

For a scalar $\xi>\inf\{\Psi(x)|x\in S\}$, we will show that the level set $L=\{x\in S|\Psi(x)\leq\xi\}$ is compact.
Since $\Psi(\bullet)$ is continuous on $S$, the set $L$ is closed and we only need to prove that $L$ is bounded.
Actually, it is guaranteed by the $l_1$ penalty term $\lambda\|\bm{\alpha}\|_1$ of $\Psi(\bullet)$.
Otherwise, unbounded points will take the penalty term go to infinity contrary to the boundedness of $\Psi(\bullet)$ on $L$.
From the compactness of the level set $L$,
the infimum $\inf\{\Psi(x)|x\in L\}$ is attainable because $\Psi(\bullet)$ is continuous on $L$.
Furthermore, since $\inf\{\Psi(x)|x\in S\}=\inf\{\Psi(x)|x\in L\}$,
there exists a point $s\in S$ such that $\Psi(s)=\inf\{\Psi(x)|x\in S\}$.
\end{proof}

\section{Low-rank approximation of tensor}
In this section, we first describe an algorithm (LRAT) of low-rank approximation of tensor for
computing the solution of problem (\ref{problem5}),
and then show some theoretical guarantees on the convergence of LRAT:
(1) The sequence $\{(\mathbf{X}^n,\mathbf{Y}^n,\mathbf{Z}^n,\bm{\alpha}^n)\}_{n\in\mathbb{N}}$ generated by LRAT
converges to a critical point of $\Psi(\bullet)$.
(2) The limit point of $\{(\mathbf{X}^n,\mathbf{Y}^n,\mathbf{Z}^n,\bm{\alpha}^n)\}_{n\in\mathbb{N}}$ is a KKT point of problem (\ref{problem4}).

\subsection{The algorithm}
\begin{algorithm}[htb]
\renewcommand{\algorithmicrequire}{\textbf{Input:}}
\renewcommand\algorithmicensure {\textbf{Output:} }
\caption{Low-Rank Approximation Of Tensors (LRAT)}
\label{alg:Framwork}
\begin{algorithmic}[1]
\REQUIRE A third order tensor $\mathcal{A}$, an upper bound $R$ of $\mbox{rank}(\mathcal{A})$, a penalty parameter $\lambda$ and a scale $s>1$;\\
\ENSURE An approximated tensor $\mathcal{\hat{B}}$ with an estimated rank $\hat{R}$;\\
\STATE Give an initial tensor $\mathcal{B}^0=[\bm{\alpha}^0;\mathbf{X}^0,\mathbf{Y}^0,\mathbf{Z}^0]_R$.
\STATE Update step: \\
b. Update matrices $\mathbf{X,Y,Z}$:\\
\hspace*{.5cm}Compute $\mathbf{U}^n$ by (\ref{update}) and let $d_n=\max\{\|\mathbf{U}^n{\mathbf{U}^n}^T\|_F,1\}$.\\
\hspace*{.5cm}Compute $\mathbf{D}^n$ and $\mathbf{X}^{n+1}$ by
\begin{align*}
&\mathbf{D}^n=\mathbf{X}^n-\frac{1}{sd_n}\nabla_\mathbf{X} f(\mathbf{X}^n,\mathbf{Y}^n,\mathbf{Z}^n,\bm{\alpha}^n),\\
&\mathbf{X}^{n+1}=\mathbf{D}^n\text{diag}(\|\mathbf{d}_1^n\|,\cdots,\|\mathbf{d}_R^n\|)^{-1}
\end{align*}
\hspace*{.5cm}where $\mathbf{d}_i^n$ is the $i$-th column of $\mathbf{D}^n$ for $i=1,\cdots,R$.\\
\hspace*{.5cm}Compute $\mathbf{V}^n$ by (\ref{update}) and let $e_n=\max\{\|\mathbf{V}^n{\mathbf{V}^n}^T\|_F,1\}$.\\
\hspace*{.5cm}Compute $\mathbf{E}^n$ and $\mathbf{Y}^{n+1}$ by
\begin{align*}
&\mathbf{E}^n=\mathbf{Y}^n-\frac{1}{se_n}\nabla_\mathbf{Y} f(\mathbf{X}^{n+1},\mathbf{Y}^n,\mathbf{Z}^n,\bm{\alpha}^n),\\
&\mathbf{Y}^{n+1}=\mathbf{E}^n\text{diag}(\|\mathbf{e}_1^n\|,\cdots,\|\mathbf{e}_R^n\|)^{-1}
\end{align*}
\hspace*{.5cm}where $\mathbf{e}_i^n$ is the $i$-th column of $\mathbf{E}^n$ for $i=1,\cdots,R$.\\
\hspace*{.5cm}Compute $\mathbf{W}^n$ by (\ref{update}) and let $f_n=\max\{\|\mathbf{W}^n{\mathbf{W}^n}^T\|_F,1\}$.\\
\hspace*{.5cm}Compute $\mathbf{F}^n$ and $\mathbf{Z}^{n+1}$ by
\begin{align*}
&\mathbf{F}^n=\mathbf{Z}^n-\frac{1}{sf_n}\nabla_\mathbf{Z} f(\mathbf{X}^{n+1},\mathbf{Y}^{n+1},\mathbf{Z}^n,\bm{\alpha}^n),\\
&\mathbf{Z}^{n+1}=\mathbf{F}^n\text{diag}(\|\mathbf{f}_1^n\|,\cdots,\|\mathbf{f}_R^n\|)^{-1}
\end{align*}
\hspace*{.5cm}where $\mathbf{f}_i^n$ is the $i$-th column of $\mathbf{F}^n$ for $i=1,\cdots,R$.\\
c. Update the row vector $\bm{\alpha}$:\\
\hspace*{.5cm}Compute $\mathbf{Q}^{n+1}$ by (\ref{update2}) and let $\eta_n=\max\{\|\mathbf{Q}^{n+1}{\mathbf{Q}^{n+1}}^T\|_F,1\}$.\\
\hspace*{.5cm}Compute $\bm{\beta}^{n+1}$ by (\ref{updatea}) and use the soft thresholding:
\hspace*{.5cm} $$\bm{\alpha}^{n+1}=\mathcal{S}_{\frac{\lambda}{s\eta_n}}(\bm{\beta}^{n+1}).$$
\STATE Denote the limitations by $\mathbf{\hat{X}},\mathbf{\hat{Y}},\mathbf{\hat{Z}},\bm{\hat{\alpha}}$,
compute $\mathcal{\hat{B}}=[\bm{\hat{\alpha}};\mathbf{\hat{X}},\mathbf{\hat{Y}},\mathbf{\hat{Z}}]_R$
and count the number $\hat{R}$ of nonzero entries in $\bm{\hat{\alpha}}$.\\
\RETURN The tensor $\mathcal{\hat{B}}$ with the estimated rank $\hat{R}$.
\end{algorithmic}
\end{algorithm}

As in (\ref{khatri1}), the matricizations of tensor $\mathcal{B}=[\bm{\alpha};\mathbf{X,Y,Z}]_R$ via Khatri-Rao products are
$$\mathbf{B}_{(1)}=\mathbf{XD}(\mathbf{Z\odot Y})^T,\mathbf{B}_{(2)}=\mathbf{YD}(\mathbf{Z\odot X})^T,
\mathbf{B}_{(3)}=\mathbf{ZD}(\mathbf{Y\odot X})^T$$ where $\mathbf{D}=diag(\alpha_1,\cdots,\alpha_R)$.
We introduce the following three matrices for updating in the Algorithm \ref{alg:Framwork}:
\begin{equation}\label{update}
\mathbf{U}=\mathbf{D}(\mathbf{Z\odot Y})^T,\mathbf{V}=\mathbf{D}(\mathbf{Z\odot X})^T,\mathbf{W}=\mathbf{D}(\mathbf{Y\odot X})^T.
\end{equation}
It follows that $\mathbf{B}_{(1)}=\mathbf{XU}$, $\mathbf{B}_{(2)}=\mathbf{YV}$ and $\mathbf{B}_{(3)}=\mathbf{ZW}$.
Thus the function $f(\mathbf{X,Y,Z,}\bm{\alpha})$ can be written in three equivalent forms: $\frac{1}{2}\|\mathbf{A}_{(1)}-\mathbf{XU}\|_F^2=
\frac{1}{2}\|\mathbf{A}_{(2)}-\mathbf{YV}\|_F^2=\frac{1}{2}\|\mathbf{A}_{(3)}-\mathbf{ZW}\|_F^2$.
Furthermore, we have the gradients of $f(\bullet)$ on $\mathbf{X,Y,Z}$:
\begin{equation}\label{updatemeth}
\begin{aligned}
\nabla_\mathbf{X} f(\mathbf{X,Y,Z,}\bm{\alpha})=(\mathbf{XU}-\mathbf{A}_{(1)})\mathbf{U}^T,\\
\nabla_\mathbf{Y} f(\mathbf{X,Y,Z,}\bm{\alpha})=(\mathbf{YV}-\mathbf{A}_{(2)})\mathbf{V}^T,\\
\nabla_\mathbf{Z} f(\mathbf{X,Y,Z,}\bm{\alpha})=(\mathbf{ZW}-\mathbf{A}_{(3)})\mathbf{W}^T.\\
\end{aligned}
\end{equation}

Using the vectorization of tensors \cite{GolubVanLoan},
we can vectorize every rank-one tensor of outer product $\mathbf{x}_r\circ\mathbf{y}_r\circ\mathbf{z}_r$ into a row vector $\mathbf{q}_r$ for $1\leq r\leq R$. We denote a matrix consisting of all $\mathbf{q}_r$ for $1\leq r\leq R$ by
\begin{equation}\label{update2}
\mathbf{Q}=(\mathbf{q}_1^T,\cdots,\mathbf{q}_R^T)^T.
\end{equation}
Thus the function $f(\mathbf{X,Y,Z,}\bm{\alpha})$ can be also written as $\frac{1}{2}\|\mathbf{a}-\bm{\alpha Q}\|_F^2$,
where $\mathbf{a}$ is a vectorization for tensor $\mathcal{A}$.
Furthermore, the gradient of $f(\bullet)$ on $\bm{\alpha}$ is
\begin{equation}
\nabla_{\bm\alpha} f(\mathbf{X,Y,Z,}\bm{\alpha})=(\bm{\alpha}\mathbf{Q}-\mathbf{a})\mathbf{Q}^T.
\end{equation}

Our algorithm starts from $(\mathbf{X}^k,\mathbf{Y}^k,\mathbf{Z}^k,\bm{\alpha}^k)$
and iteratively update variables $\mathbf{X,Y,Z}$ and then $\bm{\alpha}$ in each loop.
Inspired by the equation (\ref{updatetech}), the update of $\mathbf{X}$ is based on the following constraint optimization problem:
\begin{equation*}
\begin{aligned}
& \arg\min\limits_{\mathbf{X}} \{\langle\mathbf{X}-\mathbf{X}^n,\nabla_\mathbf{X} f(\mathbf{X}^n,\mathbf{Y}^n,\mathbf{Z}^n,\bm{\alpha}^n)\rangle+\frac{sd_n}{2}\|\mathbf{X}-\mathbf{X}^n\|_F^2\} &\\
& \quad\mbox{s.t.}\ \|\mathbf{x}_i\|=1, i=1,\cdots,R, &
\end{aligned}
\end{equation*}
where $\mathbf{X}=(\mathbf{x}_1,\cdots,\mathbf{x}_R)\in\mathbb{R}^{I\times R}$,
$d_n=\max\{\|\mathbf{U}^n{\mathbf{U}^n}^T\|_F,1\}$ and $\mathbf{U}^n$ is computed from $\bm{\alpha}^n,\mathbf{Y}^n,\mathbf{Z}^n$ by (\ref{update}).
This problem is equivalent to:
\begin{equation*}
\arg\min\limits_{\mathbf{X}} \{\|\mathbf{X}-\mathbf{D}^n\|_F^2\} \quad\mbox{s.t.}\ \|\mathbf{x}_i\|=1, i=1,\cdots,R.
\end{equation*}
where $\mathbf{D}^n=\mathbf{X}^n-\frac{1}{sd_n}\nabla_\mathbf{X} f(\mathbf{X}^n,\mathbf{Y}^n,\mathbf{Z}^n,\bm{\alpha}^n)$.
So we obtain the update of $X$:
\begin{equation*}\label{updatex}
\mathbf{x}_i^{n+1}=\mathbf{d}_i^n/\|\mathbf{d}_i^n\|, i=1,\cdots,R,
\end{equation*}
where $\mathbf{x}_i^{n+1}$ and $\mathbf{d}_i^n$ are the $i$-th columns of $\mathbf{X}^{n+1}$ and $\mathbf{D}^n$.

Similarly, the update of $\mathbf{Y}$ is based on the following optimization problem:
\begin{equation*}
\begin{aligned}
& \arg\min\limits_{\mathbf{Y}} \{\langle\mathbf{Y}-\mathbf{Y}^n,\nabla_\mathbf{Y} f(\mathbf{X}^{n+1},\mathbf{Y}^n,\mathbf{Z}^n,\bm{\alpha}^n)\rangle+\frac{se_n}{2}\|\mathbf{Y}-\mathbf{Y}^n\|_F^2\}&\\
& \quad\mbox{s.t.}\ \|\mathbf{y}_i\|=1, i=1,\cdots,R, &
\end{aligned}
\end{equation*}
where $\mathbf{Y}=(\mathbf{y}_1,\cdots,\mathbf{y}_R)\in\mathbb{R}^{J\times R}$,
$e_n=\max\{\|\mathbf{V}^n{\mathbf{V}^n}^T\|_F,1\}$ and $\mathbf{V}^n$ is computed from $\bm{\alpha}^n,\mathbf{X}^{n+1},\mathbf{Z}^n$ by (\ref{update}).
So we obtain the update of $Y$:
\begin{equation*}\label{updatey}
\mathbf{y}_i^{n+1}=\mathbf{e}_i^n/\|\mathbf{e}_i^n\|, i=1,\cdots,R,
\end{equation*}
where $\mathbf{y}_i^{n+1}$ and $\mathbf{e}_i^n$ are the $i$-th columns of $\mathbf{Y}^{n+1}$ and
$\mathbf{E}^n=\mathbf{Y}^n-\frac{1}{se_n}\nabla_\mathbf{Y} f(\mathbf{X}^{n+1},\mathbf{Y}^n,\mathbf{Z}^n,\bm{\alpha}^n)$.

The update of $\mathbf{Z}$ is based on the following constraint optimization problem:
\begin{equation*}
\begin{aligned}
& \arg\min\limits_{\mathbf{Z}} \{\langle\mathbf{Z}-\mathbf{Z}^n,\nabla_\mathbf{Z} f(\mathbf{X}^{n+1},\mathbf{Y}^{n+1},\mathbf{Z}^n,\bm{\alpha}^n)\rangle+\frac{sf_n}{2}\|\mathbf{Z}-\mathbf{Z}^n\|_F^2\}&\\
& \quad\mbox{s.t.}\ \|\mathbf{z}_i\|=1, i=1,\cdots,R, &
\end{aligned}
\end{equation*}
where $\mathbf{Z}=(\mathbf{z}_1,\cdots,\mathbf{z}_R)\in\mathbb{R}^{K\times R}$,
$f_n=\max\{\|\mathbf{W}^n{\mathbf{W}^n}^T\|_F,1\}$ and $\mathbf{W}^n$ is computed from $\bm{\alpha}^n,\mathbf{X}^{n+1},\mathbf{Y}^{n+1}$ by (\ref{update}).
The update of $Z$ is:
\begin{equation*}\label{updatez}
\mathbf{z}_i^{n+1}=\mathbf{f}_i^{n+1}/\|\mathbf{f}_i^{n+1}\|, i=1,\cdots,R,
\end{equation*}
where $\mathbf{z}_i^{n+1}$ and $\mathbf{f}_i^{n+1}$ are the $i$-th columns of $\mathbf{Z}^{n+1}$ and
$\mathbf{F}^n=\mathbf{Z}^n-\frac{1}{sf_n}\nabla_\mathbf{Z} f(\mathbf{X}^{n+1},\mathbf{Y}^{n+1},\mathbf{Z}^n,\bm{\alpha}^n)$.

Finally, we consider to update $\bm{\alpha}$ by using the equation (\ref{updatetech}):
\begin{equation*}\label{updatemeth2}
\arg\min\limits_{\bm{\alpha}} \{\langle\bm{\alpha}-\bm{\alpha}^n,\nabla_{\bm{\alpha}} f(\mathcal{C}^{n+1},\mathbf{X}^{n+1},\mathbf{Y}^{n+1},\mathbf{Z}^{n+1},\bm{\alpha}^n)\rangle+\frac{s\eta_n}{2}\|\bm{\alpha}-\bm{\alpha}^n\|^2+\lambda\|\bm{\alpha}\|_1\}.
\end{equation*}
where $\eta_n=\max\{\|\mathbf{Q}^{n+1}{\mathbf{Q}^{n+1}}^T\|_F,1\}$ and $\mathbf{Q}^{n+1}$ can be computed from $\mathbf{X}^{n+1},\mathbf{Y}^{n+1},\mathbf{Z}^{n+1}$ by (\ref{update2}).
This optimization problem is equivalent to:
\begin{equation*}
\arg\min\limits_{\bm{\alpha}} \frac{1}{2}\|\bm{\alpha}-\bm{\alpha}^n+\frac{1}{s\eta_n}\nabla_{\bm{\alpha}} f(\mathcal{C}^{n+1},\mathbf{X}^{n+1},\mathbf{Y}^{n+1},\mathbf{Z}^{n+1},\bm{\alpha}^n)\|^2
+\frac{\lambda}{s\eta_n}\|\bm{\alpha}\|_1.
\end{equation*}

So we can obtain the update form for $\bm{\alpha}$ in Algorithm \ref{alg:Framwork}
by using the separate soft thresholding:
\begin{equation*}
\bm{\alpha}^{n+1}=\mathcal{S}_{\frac{\lambda}{s\eta_n}}(\bm{\beta}^{n+1})
\end{equation*}
where
\begin{equation}\label{updatea}
\bm{\beta}^{n+1}=\bm{\alpha}^n-\frac{1}{s\eta_n}\nabla_{\bm{\alpha}} f(\mathcal{C}^{n+1},\mathbf{X}^{n+1},\mathbf{Y}^{n+1},\mathbf{Z}^{n+1},\bm{\alpha}^n).
\end{equation}

It should be noted that if we set $\lambda=0$, the LRAT algorithm turns into a
modified alternative least square method (modALS).
This modified ALS algorithm uses linearized iterative technique \cite{bolte,xu} to update variables in each step.
Although the regularization parameter $\lambda$ is fixed in Algorithm \ref{alg:Framwork},
we can adaptively choose it for practical computation, which will be shown in the Section 5.

\begin{remark}
In our algorithm, the computational complexity mainly comes from matrix multiplications. The Update Step (2b) for updating $\bm{\alpha}$ in LRAT Algorithm require more cpu time than the Update Step (2a) because of the large matrix dimension of $\mathbf{Q}$. The complexity of our algorithm is $O(NIJKR^2)$, where $N$ is the total number of iteration.
\end{remark}

\subsection{Convergence of algorithm}
In this subsection, we illustrate the convergence mechanism of the LRAT algorithm,
which is a rescaling version of the proximal alternating linear minimization algorithm \cite{bolte}.
The following Lemma \ref{gradcon} points out that
for the function $f(\bm{\omega})=\frac{1}{2}\|\mathcal{A}-[\bm{\alpha};\mathbf{X,Y,Z}]_R\|^2_F$ ,
the gradient $\nabla_{\bm{\omega}}f(\bm{\omega})$ of $f(\bm{\omega})$
is Lipschitz continuous on bounded subsets and
all the partial gradients of $f(\bm{\omega})$ are globally Lipschitz with modulus.

\begin{lemma}\label{gradcon}
Let $f(\bm{\omega})$ be the approximation term $\frac{1}{2}\|\mathcal{A}-[\bm{\alpha};\mathbf{X,Y,Z}]_R\|^2_F$
where $\bm{\omega}=(\mathbf{X},\mathbf{Y},\mathbf{Z},\bm{\alpha})$.
We have that the gradient function $\nabla f$ is Lipschitz continuous on bounded subsets of
$\mathbb{R}^{I\times R}\times\mathbb{R}^{J\times R}\times\mathbb{R}^{K\times R}\times\mathbb{R}^R$, i.e.,
for any bounded subset $B\in\mathbb{R}^{I\times R}\times\mathbb{R}^{J\times R}\times\mathbb{R}^{K\times R}\times\mathbb{R}^R$,
there exists $M>0$ such that for any $\bm{\omega}_1,\bm{\omega}_2\in B$,
$$\|\nabla_{\bm{\omega}}f(\bm{\omega}_1)-\nabla_{\bm{\omega}} f(\bm{\omega}_2)\|_F\leq M\|\bm{\omega}_1-\bm{\omega}_2\|_F.$$
Moreover, for any fixed $\mathbf{X}\in\mathbb{R}^{I\times R},\mathbf{Y}\in\mathbb{R}^{J\times R},\mathbf{Z}\in\mathbb{R}^{K\times R},\bm{\alpha}\in\mathbb{R}^{R}$, there exist four constants $c,d,e,\eta>0$ such that:
\begin{align*}
\|\nabla_\mathbf{X} f(\mathbf{X}_1,\mathbf{Y},\mathbf{Z},\bm{\alpha})-\nabla_\mathbf{X} f(\mathbf{X}_2,\mathbf{Y},\mathbf{Z},\bm{\alpha})\|_F\leq d\|\mathbf{X}_1-\mathbf{X}_2\|_F, \mbox{for any}\; \mathbf{X}_1,\mathbf{X}_2\in\mathbb{R}^{I\times R}\\
\|\nabla_\mathbf{Y} f(\mathbf{X},\mathbf{Y}_1,\mathbf{Z},\bm{\alpha})-\nabla_\mathbf{Y} f(\mathbf{X},\mathbf{Y}_2,\mathbf{Z},\bm{\alpha})\|_F\leq e\|\mathbf{Y}_1-\mathbf{Y}_2\|_F,\mbox{for any}\; \mathbf{Y}_1,\mathbf{Y}_2\in\mathbb{R}^{J\times R}\\
\|\nabla_\mathbf{Z} f(\mathbf{X},\mathbf{Y},\mathbf{Z}_1,\bm{\alpha})-\nabla_\mathbf{Z} f(\mathbf{X},\mathbf{Y},\mathbf{Z}_2,\bm{\alpha})\|_F\leq f\|\mathbf{Z}_1-\mathbf{Z}_2\|_F,\mbox{for any}\; \mathbf{Z}_1,\mathbf{Z}_2\in\mathbb{R}^{K\times R}\\
\|\nabla_{\bm{\alpha}} f(\mathbf{X},\mathbf{Y},\mathbf{Z},\bm{\alpha}_1)-\nabla_{\bm{\alpha}} f(\mathbf{X},\mathbf{Y},\mathbf{Z},\bm{\alpha}_2)\|_F\leq \eta\|\bm{\alpha}_1-\bm{\alpha}_2\|_F,\mbox{for any}\; \bm{\alpha}_1,\bm{\alpha}_2\in\mathbb{R}^{R}
\end{align*}
where $d=\|\mathbf{UU^T}\|_F,e=\|\mathbf{VV^T}\|_F,f=\|\mathbf{WW^T}\|_F,\eta=\|\mathbf{QQ^T}\|_F$.\\
\end{lemma}
The proof has not been included since it relies on standard techniques.
In our LRAT algorithm, those Lipschitz constants rely on the iterative number $n$ and have a lower bound $1$. Specifically,
$d_n=\max\{\|\mathbf{U}^{n+1}{\mathbf{U}^{n+1}}^T\|_F,1\}$, $e_n=\max\{\|\mathbf{V}^{n+1}{\mathbf{V}^{n+1}}^T\|_F,1\}$,
$f_n=\max\{\|\mathbf{W}^{n+1}{\mathbf{W}^{n+1}}^T\|_F,1\}$, $\eta_n=\max\{\|\mathbf{Q}^{n+1}{\mathbf{Q}^{n+1}}^T\|_F,1\}$.

\begin{lemma}\label{sufdec}(Sufficient decrease property \cite{bolte})
Let $f:\mathbb{R}^m\rightarrow\mathbb{R}$ be a continuously differentiable with gradient $\nabla f$ assumed $L_f$-Lipschitz continuous
and let $\sigma:\mathbb{R}^m\rightarrow(-\infty,+\infty]$ be a proper and lower semicontinuous function with $\inf_{\mathbb{R}^m}\sigma>-\infty$.
For any $t>L_f$ and $u\in\text{dom}\ \sigma$, define
$$u^+=\arg\min\limits_{x} \{\langle x-u,\nabla f(u)\rangle+\frac{t}{2}\|x-u\|^2+\sigma(u)\}.$$
Then we have that
\begin{equation}
f(u)+\sigma(u)-(f(u^+)+\sigma(u^+))\geq\frac{1}{2}(t-L_f)\|u^+-u\|^2.
\end{equation}
\end{lemma}
\begin{lemma}\label{descent}
Let $\Psi(\bullet)$ be the objective function in problem (\ref{problem5}).
If $(\mathbf{X}^n,\mathbf{Y}^n,\mathbf{Z}^n,\bm{\alpha}^n)_{n\in\mathbb{N}}$ and $(d_n,e_n,f_n,\eta_n)_{n\in\mathbb{N}}$ are generated by our LRAT algorithm,
we have that for any $s>1$
\begin{align*}
\Psi(\mathbf{X}^n,\mathbf{Y}^n,\mathbf{Z}^n,\bm{\alpha}^n)-\Psi(\mathbf{X}^{n+1},\mathbf{Y}^n,\mathbf{Z}^n,\bm{\alpha}^n)\geq\frac{1}{2}(s-1)d_n\|\mathbf{X}^n-\mathbf{X}^{n+1}\|_F^2, \\
\Psi(\mathbf{X}^{n+1},\mathbf{Y}^n,\mathbf{Z}^n,\bm{\alpha}^n)-\Psi(\mathbf{X}^{n+1},\mathbf{Y}^{n+1},\mathbf{Z}^n,\bm{\alpha}^n)\geq\frac{1}{2}(s-1)e_n\|\mathbf{Y}^n-\mathbf{Y}^{n+1}\|_F^2,  \\
\Psi(\mathbf{X}^{n+1},\mathbf{Y}^{n+1},\mathbf{Z}^n,\bm{\alpha}^n)-\Psi(\mathbf{X}^{n+1},\mathbf{Y}^{n+1},\mathbf{Z}^{n+1},\bm{\alpha}^n)\geq\frac{1}{2}(s-1)f_n\|\mathbf{Z}^n-\mathbf{Z}^{n+1}\|_F^2,\\
\Psi(\mathbf{X}^{n+1},\mathbf{Y}^{n+1},\mathbf{Z}^{n+1},\bm{\alpha}^n)-\Psi(\mathbf{X}^{n+1},\mathbf{Y}^{n+1},\mathbf{Z}^{n+1},\bm{\alpha}^{n+1})\geq\frac{1}{2}(s-1)\eta_n\|\bm{\alpha}^n-\bm{\alpha}^{n+1}\|_F^2.
\end{align*}
\end{lemma}
\begin{proof}
These four inequalities can be obtained by using Lemma \ref{sufdec}.
\end{proof}

The following lemma shows that the value of $\Psi(\bullet)$ monotonically decreases on the
sequence $(\bm{\omega}^{n})_{n\in\mathbb{N}}$, which is generated by our algorithm.

\begin{lemma}\label{keylem}
Let $\Psi(\bm{\omega})$ be the objective function $$\frac{1}{2}\|\mathcal{A}-[\bm{\alpha};\mathbf{X,Y,Z}]_R\|^2_F+\lambda\|\bm{\alpha}\|_1+\delta_{S_1}(\mathbf{X})+\delta_{S_2}(\mathbf{Y})+\delta_{S_3}(\mathbf{Z})$$
where $\bm{\omega}=(\mathbf{X},\mathbf{Y},\mathbf{Z},\bm{\alpha})$, then\\
$(i)$ the sequence $\{\Psi(\bm{\omega}^n)\}_{n\in\mathbb{N}}$ is nonincreasing and for any $n\in\mathbb{N}$,
there is a scalar $\beta>0$ such that $\Psi(\bm{\omega}^{n})-\Psi(\bm{\omega}^{n+1})\geq
\beta\|\bm{\omega}^n-\bm{\omega}^{n+1}\|_F^2$.\\
$(ii)$$\lim\limits_{n\rightarrow\infty}\|\mathbf{X}^n-\mathbf{X}^{n+1}\|_F\rightarrow 0$,$\lim\limits_{n\rightarrow\infty}\|\mathbf{Y}^n-\mathbf{Y}^{n+1}\|_F\rightarrow 0$,$\lim\limits_{n\rightarrow\infty}\|\mathbf{Z}^n-\mathbf{Z}^{n+1}\|_F\rightarrow 0$ and $\lim\limits_{n\rightarrow\infty}\|\bm{\alpha}^n-\bm{\alpha}^{n+1}\|_F\rightarrow 0$.\\
$(iii)$ the sequence $\{\bm{\omega}^n\}_{n\in\mathbb{N}}$ is bounded.\\
\end{lemma}
\begin{proof}
In our algorithm, all the Lipschitz constants $d_n,e_n,f_n,\eta_n\geq 1$.
So by Lemma \ref{descent}, $\Psi(\bm{\omega}^{n})-\Psi(\bm{\omega}^{n+1})
\geq\beta\|\bm{\omega}^n-\bm{\omega}^{n+1}\|_F^2$ where $\beta=\min\{(s-1)/2,1/2\}$.
We can obtain the first conclusion $(i)$.

The second conclusion $(ii)$ holds from the first one because the sum
$\sum\limits_{n=0}^{\infty}(\Psi(\bm{\omega}^{n})-\Psi(\bm{\omega}^{n+1}))$ is finite.

If the sequence $\{\bm{\omega}^n\}_{n\in\mathbb{N}}$ is unbounded,
it means that $\{\bm{\alpha}^n\}_{n\in\mathbb{N}}$ is unbounded
since columns of $\mathbf{X}^n,\mathbf{Y}^n,\mathbf{Z}^n$ are constrained to have length one.
So the sequence $\{\Psi(\bm{\omega}^n)\}_{n\in\mathbb{N}}$ is unbounded since $\Psi(\mathcal{C},\mathbf{X,Y,Z,}\bm{\alpha})\geq\lambda\|\bm{\alpha}\|_1$.
From the conclusion $(i)$, $\Psi(\bm{\omega}^n)$ is nonincreasing.
Since $\Psi(\bullet)$ has a lower bound, the sequence $\{\Psi(\bm{\omega}_n)\}_{n\in\mathbb{N}}$ is not unbounded.
It is a contradiction. So the sequence $\{\bm{\omega}^n\}_{n\in\mathbb{N}}$ must be bounded.
\end{proof}

Furthermore, from Lemma \ref{gradcon} and the boundness shown in Lemma \ref{keylem},
we can obtain the following Lipschitz upper bounds for subdifferentials.
\begin{lemma}\label{gradbound}
Let $\bm{\omega}^n=(\mathbf{X}^n,\mathbf{Y}^n,\mathbf{Z}^n,\bm{\alpha}^n)$ be the sequence generated by our LRAT algorithm.
There exist four positive scales $L_1,L_2,L_3$ and $L_4$ such that the following inequalities hold for any $n\in\mathbb{N}$. \\
There is some $\bm{\eta}_1^n\in\partial_{\mathbf{X}}\Psi(\bm{\omega}^{n})$
such that $\|\bm{\eta}_1^n\|_F\leq L_1\|\bm{\omega}^{n}-\bm{\omega}^{n-1}\|_F$.\\
There is some $\bm{\eta}_2^n\in\partial_{\mathbf{Y}}\Psi(\bm{\omega}^{n})$
such that $\|\bm{\eta}_2^n\|_F\leq L_2\|\bm{\omega}^{n}-\bm{\omega}^{n-1}\|_F$.\\
There is some $\bm{\eta}_3^n\in\partial_{\mathbf{Z}}\Psi(\bm{\omega}^{n})$
such that $\|\bm{\eta}_3^n\|_F\leq L_3\|\bm{\omega}^{n}-\bm{\omega}^{n-1}\|_F$.\\
There is some $\bm{\eta}_4^n\in\partial_{\bm{\alpha}}\Psi(\bm{\omega}^n)$
such that $\|\bm{\eta}_4^n\|_F\leq L_4\|\bm{\omega}^{n}-\bm{\omega}^{n-1}\|_F$.
\end{lemma}
\begin{proof}
By the update of $\mathbf{X}$,
\begin{equation*}
\begin{aligned}
\mathbf{X}^n=\arg\min\limits_{\mathbf{X}} \{ & \langle\mathbf{X}-\mathbf{X}^{n-1},\nabla_\mathbf{X} f(\mathcal{C}^{n},\mathbf{X}^{n-1},\mathbf{Y}^{n-1},\mathbf{Z}^{n-1},\bm{\alpha}^{n-1})\rangle\\
&+\frac{sd_n}{2}\|\mathbf{X}-\mathbf{X}^{n-1}\|_F^2+\delta_{S_1}(\mathbf{X})\}.
\end{aligned}
\end{equation*}
So we have that
$\nabla_{\mathbf{X}}f(\mathbf{X}^{n-1},\mathbf{Y}^{n-1},\mathbf{Z}^{n-1},\bm{\alpha}^{n-1})+sd_n(\mathbf{X}^n-\mathbf{X}^{n-1})+\mathbf{u}_1^n=\mathbf{0}$
where $\mathbf{u}_1^n\in\partial_{\mathbf{X}}\delta_{S_1}(\mathbf{X}^n)$.
Hence
$$\mathbf{u}_1^n=sd_n(\mathbf{X}^{n-1}-\mathbf{X}^{n})-\nabla_{\mathbf{X}}f(\mathbf{X}^{n-1},\mathbf{Y}^{n-1},\mathbf{Z}^{n-1},\bm{\alpha}^{n-1}).$$
Since $\nabla_{\mathbf{X}}f(\mathbf{X}^{n},\mathbf{Y}^{n},\mathbf{Z}^{n},\bm{\alpha}^{n})+\mathbf{u}_1^n\in\partial_{\mathbf{X}}\Psi(\mathbf{X}^n,\mathbf{Y}^n,\mathbf{Z}^n,\bm{\alpha}^n)$,
we have that
\begin{equation}
\begin{aligned}
\bm{\eta}_1^n=& \nabla_{\mathbf{X}}f(\mathbf{X}^{n},\mathbf{Y}^{n},\mathbf{Z}^{n},\bm{\alpha}^{n})+
sd_n(\mathbf{X}^{n-1}-\mathbf{X}^{n})\\
& -\nabla_{\mathbf{X}}f(\mathbf{X}^{n-1},\mathbf{Y}^{n-1},\mathbf{Z}^{n-1},\bm{\alpha}^{n-1})\\
= & \nabla_{\mathbf{X}}f(\mathbf{X}^{n},\mathbf{Y}^{n},\mathbf{Z}^{n},\bm{\alpha}^{n})+\mathbf{u}_1^n\\
\in & \partial_{\mathbf{X}}\Psi(\mathbf{X}^n,\mathbf{Y}^n,\mathbf{Z}^n,\bm{\alpha}^n)
\end{aligned}
\end{equation}
By Lemma \ref{gradcon} and the boundness of $\{\bm{\omega}^n\}_{n\in\mathbb{N}}$,
we have that there exists a constant $L_1$ such that $\|\bm{\eta}_1^n\|_F\leq L_1\|\bm{\omega}^{n}-\bm{\omega}^{n-1}\|_F$.

Similarly, we can choose
\begin{equation}
\begin{aligned}
\bm{\eta}_2^n=& \nabla_{\mathbf{Y}}f(\mathbf{X}^{n},\mathbf{Y}^{n},\mathbf{Z}^{n},\bm{\alpha}^{n})+
se_n(\mathbf{Y}^{n-1}-\mathbf{Y}^{n})\\
& -\nabla_{\mathbf{Y}}f(\mathbf{X}^{n},\mathbf{Y}^{n-1},\mathbf{Z}^{n-1},\bm{\alpha}^{n-1})\\
\end{aligned}
\end{equation}
and
\begin{equation}
\begin{aligned}
\bm{\eta}_3^n=& \nabla_{\mathbf{Z}}f(\mathbf{X}^{n},\mathbf{Y}^{n},\mathbf{Z}^{n},\bm{\alpha}^{n})+
sf_n(\mathbf{Z}^{n-1}-\mathbf{Z}^{n})\\
& -\nabla_{\mathbf{Z}}f(\mathbf{X}^{n},\mathbf{Y}^{n},\mathbf{Z}^{n-1},\bm{\alpha}^{n-1})\\
\end{aligned}
\end{equation}
So $\bm{\eta}_2^n\in\partial_{\mathbf{Y}}\Psi(\mathbf{X}^n,\mathbf{Y}^n,\mathbf{Z}^n,\bm{\alpha}^n)$ and
$\bm{\eta}_3^n\in\partial_{\mathbf{Z}}\Psi(\mathbf{X}^n,\mathbf{Y}^n,\mathbf{Z}^n,\bm{\alpha}^n)$.
Furthermore, there exist constants $L_2$ and $L_3$ such that
$\|\bm{\eta}_2^n\|_F\leq L_2\|\bm{\omega}^{n}-\bm{\omega}^{n-1}\|_F$
and $\|\bm{\eta}_3^n\|_F\leq L_3\|\bm{\omega}^{n}-\bm{\omega}^{n-1}\|_F$.

By the update of $\bm{\alpha}$,
\begin{equation}\label{gradona}
\nabla_{\bm{\alpha}}f(\mathbf{X}^n,\mathbf{Y}^n,\mathbf{Z}^n,\bm{\alpha}^{n-1})+
s\eta_n(\bm{\alpha}^n-\bm{\alpha}^{n-1})+\mathbf{u}^n=0
\end{equation}
where $\mathbf{u}^n\in\partial_{\bm{\alpha}}g(\bm{\alpha}^n)$
and $g(\mathbf{x})=\lambda\|\mathbf{x}\|_1$.
Denote $\bm{\eta}_4^n$ as $\nabla_{\bm{\alpha}}f(\mathbf{X}^n,\mathbf{Y}^n,\mathbf{Z}^n,\bm{\alpha}^n)+\mathbf{u}^n$.
Thus we have that $\bm{\eta}_4^n\in
\partial_{\bm{\alpha}}\Psi(\mathbf{X}^n,\mathbf{Y}^n,\mathbf{Z}^n,\bm{\alpha}^n)$ and
\begin{align*}
\|\bm{\eta}_4^n\|_F&=\|\nabla_{\bm{\alpha}}f(\mathbf{X}^n,\mathbf{Y}^n,\mathbf{Z}^n,\bm{\alpha}^n)+\mathbf{u}^n\|_F\\
&\leq\|\nabla_{\bm{\alpha}}f(\mathbf{X}^n,\mathbf{Y}^n,
\mathbf{Z}^n,\bm{\alpha}^n)-
\nabla_{\bm{\alpha}}f(\mathbf{X}^n,\mathbf{Y}^n,\mathbf{Z}^n,\bm{\alpha}^{n-1})\|_F
+s\eta_n\|\bm{\alpha}^n-\bm{\alpha}^{n-1}\|_F\\
&\leq L_4\|\bm{\omega}^{n}-\bm{\omega}^{n-1}\|_F
\end{align*}
We also get the last inequality by using Lemma \ref{gradcon} and the boundness of $\{\bm{\omega}^n\}_{n\in\mathbb{N}}$.
\end{proof}

The following theorem shows that the sequence of the LRAT algorithm is convergent to a critical point of $\Psi(\bullet)$.
\begin{theorem}\label{limitpoint}
Let $\{\bm{\omega}^n\}_{n\in\mathbb{N}}$ be a sequence generated by the LRAT algorithm from a starting point $\bm{\omega}^0$.
Then the sequence $\{\bm{\omega}^n\}_{n\in\mathbb{N}}$ converges to a critical point $\bm{\omega}^*$ of $\Psi(\bm{\omega})$.
\end{theorem}
\begin{proof}
By Lemma \ref{descent}, the sufficient decrease property is satisfied that there is a constant $\beta>0$
such that for any $n\in\mathbb{N}$
$$\beta\|\bm{\omega}^n-\bm{\omega}^{n+1}\|_F^2\leq\Psi(\bm{\omega}^{n})-\Psi(\bm{\omega}^{n+1}).$$

By Lemma \ref{gradbound}, the iterates gap has a lower bound by the length of a vector in the subdifferential of $\Psi$.
There is a constant $L>0$ and $\{\bm{\eta}^n\}_{n\in\mathbb{N}}$ such that for any $n\in\mathbb{N}$,
$$\|\bm{\eta}^n\|_F\leq L\|\bm{\omega}^{n}-\bm{\omega}^{n-1}\|_F$$
where $\bm{\eta}^n\in\partial\Psi(\bm{\omega}^{n})$.

Furthermore, since $\Psi(\bullet)$ is a KL function, we complete the proof by using Theorem 3.1 in \cite{bolte}.
\end{proof}

A point $\bm{\omega}=(\mathbf{X},\mathbf{Y},\mathbf{Z},\bm{\alpha})$ is called as a KKT point of problem (\ref{problem4})
if there are three diagonal matrices $\mathbf{H}_1,\mathbf{H}_2,\mathbf{H}_3\in\mathbb{R}^{R\times R}$ and
a vector $\mathbf{u}\in\partial_{\bm{\alpha}}g(\bm{\alpha})$
such that
\begin{equation}
\begin{aligned}
\nabla_{\mathbf{X}}f(\bm{\omega})+\mathbf{X}\mathbf{H}_1=\mathbf{0}, \nabla_{\mathbf{Y}}f(\bm{\omega})+\mathbf{Y}\mathbf{H}_2=\mathbf{0}, \nabla_{\mathbf{Z}}f(\bm{\omega})+\mathbf{Z}\mathbf{H}_3=\mathbf{0}\\
\nabla_{\bm{\alpha}}f(\bm{\omega})+\mathbf{u}=\mathbf{0},\mathbf{N(X,Y,Z)}=1.\\
\end{aligned}
\end{equation}
In the following, we show that the limit point $\bm{\omega}^*=(\mathcal{C}^*,\mathbf{X}^*,\mathbf{Y}^*,\mathbf{Z}^*,\bm{\alpha}^*)$
of the sequence $\{\bm{\omega}^n\}_{n\in\mathbb{N}}$ is a KKT point of problem (\ref{problem4}).

\begin{corollary}
Let $\bm{\omega}^*=(\mathbf{X}^*,\mathbf{Y}^*,\mathbf{Z}^*,\bm{\alpha}^*)$ be the limit point
of the sequence $\{\bm{\omega}^n\}_{n\in\mathbb{N}}$ generated by the LRAT algorithm.
If $\mathbf{X}^*,\mathbf{Y}^*$ and $\mathbf{Z}^*$ has full column rank,
the limit point $\bm{\omega}^*$ is a KKT point of problem (\ref{problem4}).
\end{corollary}
\begin{proof}
$\mathbf{N(X^*,Y^*,Z^*)}=1$ is obvious
since $\mathbf{N(X^n,Y^n,Z^n)}=1$ and the convergency of $\{\bm{\omega}^n\}_{n\in\mathbb{N}}$.
From (\ref{gradona}), there exists a vector $\mathbf{u}\in\partial_{\bm{\alpha}}g(\bm{\alpha}^*)$ such that
$$\nabla_{\bm{\alpha}}f(\bm{\omega}^*)+\mathbf{u}=\mathbf{0}.$$
By the update of $\mathbf{X}$,
there is a diagonal matrix $\mathbf{H}_1^n$ such that
\begin{equation*}
\nabla_{\mathbf{X}}f(\mathbf{X}^{n-1},\mathbf{Y}^{n-1},\mathbf{Z}^{n-1},\bm{\alpha}^{n-1})+sd_n(\mathbf{X}^n-\mathbf{X}^{n-1})+\mathbf{X}^n\mathbf{H}_1^n=\mathbf{0}.
\end{equation*}
By the convergency of $\{\bm{\omega}^n\}_{n\in\mathbb{N}}$,
we have that $\mathbf{H}_1^n$ is convergent to some diagonal matrix $\mathbf{H}_1^*$ since $\mathbf{X}^*$ has full column rank.
Furthermore, we can obtain
\begin{equation*}
\nabla_{\mathbf{X}}f(\mathbf{X}^{*},\mathbf{Y}^{*},\mathbf{Z}^{*},\bm{\alpha}^{*})+\mathbf{X}^*\mathbf{H}_1^*=\mathbf{0}.
\end{equation*}
Similarly, we have
\begin{equation*}
\begin{aligned}
\nabla_{\mathbf{Y}}f(\mathbf{X}^{*},\mathbf{Y}^{*},\mathbf{Z}^{*},\bm{\alpha}^{*})+\mathbf{Y}^*\mathbf{H}_2^*=\mathbf{0},
\nabla_{\mathbf{Z}}f(\mathbf{X}^{*},\mathbf{Y}^{*},\mathbf{Z}^{*},\bm{\alpha}^{*})+\mathbf{Z}^*\mathbf{H}_3^*=\mathbf{0}.
\end{aligned}
\end{equation*}
This completes the proof of this corollary.
\end{proof}
\section{Probabilistic consistency of the sparsity}
In this section, we will discuss the probabilistic consistency of the sparsity
of the optimal solution to problem (\ref{problem4}).
We will see that under a suitable choice on the regularization parameter,
the optimal solution can recover the true sparsity in a statistical model with a high probability.

For a given regularization parameter $\lambda>0$,
an optimal solution to problem (\ref{problem4})
is denoted by
$$(\mathbf{\hat{X},\hat{Y},\hat{Z}},\bm{\hat{\alpha}})=
\argmin\limits_{\mathbf{X,Y,Z},\bm{\alpha}}\frac{1}{2}\|\mathcal{A}-[\bm{\alpha};\mathbf{X,Y,Z}]_R\|_F^2+\lambda\|\bm{\alpha}\|_1 \quad\mbox{s.t.}\ \mathbf{N(X,Y,Z)}=1.$$
As shown in Section 4.1, we can construct a $R\times(I*J*K)$ matrix $\mathbf{\hat{Q}}=(\mathbf{\hat{q}}_1^T,\cdots,\mathbf{\hat{q}}_R^T)^T=((\mathbf{\hat{X}}\odot\mathbf{\hat{Y}})\odot\mathbf{\hat{Z}})^T$
from (\ref{update2}), and vectorize tensor $\mathcal{A}$ into a row vector $\mathbf{a}$.

For convenience, we introduce new variables: $\mathbf{b},\bm{\theta},\mathbf{B}$ for
$\mathbf{a}^T,\bm{\alpha}^T,\mathbf{\hat{Q}}^T$ respectively.
Thus $\mathbf{b}$ and $\bm{\theta}$ are column vectors with dimension $I*J*K$ and $R$,
and $\mathbf{B}$ is a $(I*J*K)\times R$ matrix.
Furthermore, we have the following equality
\begin{equation}
\frac{1}{2}\|\mathcal{A}-[\bm{\alpha};\mathbf{\hat{X},\hat{Y},\hat{Z}}]_R\|_F^2+\|\bm{\alpha}\|_1
=\frac{1}{2}\|\mathbf{b}-\mathbf{B}\bm{\theta}\|_2^2+\lambda\|\bm{\theta}\|_1.
\end{equation}
The optimal solution $\bm{\hat{\alpha}}^T$ for tensor approximation problem (\ref{problem4})
is also an optimal solution $\bm{\hat{\theta}}$ of a standard $l_1$-regularized least square problem
\begin{equation}\label{objectarg}
\mathop{\min}\limits_{\bm{\theta}}\frac{1}{2}\|\mathbf{b}-\mathbf{B}\bm{\theta}\|_2^2+\lambda\|\bm{\theta}\|_1.
\end{equation}
Assume that $\mathbf{b}$ and $\mathbf{B}$ have a sparse representation structure as
\begin{equation}\label{sparsity}
\mathbf{b}=\mathbf{B}\bm{\theta}^*+\bm{\varepsilon},
\end{equation}
where all the columns of $\mathbf{B}$ are normalized to one. The variable $\bm{\theta}^*$ is a sparse signal with $k$ non-zero entries $(k<R)$,
and $\bm{\varepsilon}$ is a vector with independent subgaussian entries of mean zero and parameter $\sigma^2$.

Denote a subgradient vector in $\partial\|\bm{\theta}\|_1$ as $\bm{\beta}=(\beta_1,\cdots,\beta_R)^T$.
The entries of $\bm{\beta}$ satisfy that for any $1\leq i\leq R$,
$\beta_i=\mbox{sgn}(\theta_i)$ if $\theta_i\neq 0$ and $\beta_i\in[-1,1]$ if $\theta_i=0$.
As shown in the Lemma 1 of \cite{wainwright}, $\bm{\hat{\theta}}$ is an optimal solution to problem (\ref{objectarg})
if and only if there exists a subgradient vector $\bm{\hat{\beta}}\in\partial\|\bm{\hat{\theta}}\|_1$ such that
\begin{equation}
-\mathbf{B}^T(\mathbf{b}-\mathbf{B}\bm{\hat{\theta}})+\lambda\bm{\hat{\beta}}=\mathbf{0}
\end{equation}
if and only if there exists a subgradient vector $\bm{\hat{\beta}}\in\partial\|\bm{\hat{\theta}}\|_1$ such that
\begin{equation}\label{lassocon}
\mathbf{B}^T\mathbf{B}(\bm{\hat{\theta}}-\bm{\theta}^*)-\mathbf{B}^T\bm{\varepsilon}+\lambda\bm{\hat{\beta}}=\mathbf{0}.
\end{equation}

Assume that $\mathbf{B}$ is a full column rank matrix.
Then the objective function in problem (\ref{objectarg}) is strictly convex,
and the optimal solution $\bm{\hat{\theta}}$ to problem (\ref{objectarg})
is unique and exact $\bm{\hat{\alpha}}^T$.
Denote $S$ and $\hat{S}$ as the index sets of non-zero entries
in $\bm{\theta}^*$ and $\bm{\hat{\theta}}$ respectively.
So the sparse signal $\bm{\theta}^*$ can be rewritten as $({\bm{\theta}_S^*}^T,\mathbf{0}^T)^T$ and the cardinality of $S$ is $k$.
We will show in Theorem \ref{probcon} that the optimal solution $\bm{\hat{\theta}}$, which is also the $\bm{\hat{\alpha}}^T$,
of problem (\ref{objectarg})
may become a suitable approximation for the real sparse
signal $\bm{\theta}^*$.
Similar results shown in \cite{wainwright,zhao} consider the case
$\mathbf{B}^T\mathbf{B}/n\rightarrow\mathbf{C}$ as $n\rightarrow\infty$ or $n^{-1/2}\max_{j\in S^c}\|\mathbf{B}_j\|\leq 1$ where $n$ is the number of rows in $\mathbf{B}$,
while in this paper all the $\mathbf{B}_j$ are normalized to one.
We can further obtain a specific probability bound shown in Theorem \ref{probcon}, which relies only on two intrinsic parameters of model.

According to the unknown set $S$,
we can separate columns of the design matrix $\mathbf{B}$
as two parts $(\mathbf{B}_S,\mathbf{B}_{S^C})$, where $S^C$ is the complement of $S$.
Moreover, since $\mathbf{B}_S$ also have full column rank,
there exists a unique solution $\bm{\hat{\theta}}_S$ by solving the restricted Lasso problem:\\
\begin{equation}\label{locop}
\mathop{\min}\limits_{\bm{\theta}_S}\frac{1}{2}\|\mathbf{b}-\mathbf{B}_S\bm{\theta}_S\|_2^2+\lambda\|\bm{\theta}_S\|_1.
\end{equation}
If furthermore $({\bm{\hat{\theta}}_S}^T,\mathbf{0}^T)^T$ satisfies the equation (\ref{lassocon}),
thus $({\bm{\hat{\theta}}_S}^T,\mathbf{0}^T)^T$ is the unique optimal solution $\bm{\hat{\theta}}$ to problem (\ref{objectarg})
since $\mathbf{B}$ has full column rank.
Moreover, we also obtain that the index set $\hat{S}\subseteq S$.
From (\ref{lassocon}), if $\bm{\hat{\theta}}_S$ satisfies two equations:
\begin{equation}\label{req1}
\mathbf{B}_S^T\mathbf{B}_S(\bm{\hat{\theta}}_S-\bm{\theta}_S^*)-\mathbf{B}_S^T\bm{\varepsilon}+\lambda\bm{\hat{\beta}}_S=\mathbf{0}
\end{equation} and
\begin{equation}\label{req2}
\mathbf{B}_{S^C}^T\mathbf{B}_S(\bm{\hat{\theta}}_S-\bm{\theta}_S^*)-\mathbf{B}_{S^C}^T\bm{\varepsilon}+\lambda\bm{\hat{\beta}}_{S^C}=\mathbf{0},
\end{equation}
where $\bm{\hat{\beta}}_S\in\partial\|\bm{\hat{\theta}}_S\|_1$ and
$\|\bm{\hat{\beta}}_{S^C}\|_{\infty}=\max\limits_{j\in S^C}|\bm{\hat{\beta}}_{j}|<1$,
we have that
$\bm{\hat{\theta}}=({\bm{\hat{\theta}}_S}^T,\mathbf{0}^T)^T$ satisfies the equation (\ref{lassocon})
and $(\bm{\hat{\beta}}_S^T,\bm{\hat{\beta}}_{S^C}^T)^T\in\partial\|\bm{\hat{\theta}}\|_1$.
Actually, since $\bm{\hat{\theta}}_S$ minimizes the problem (\ref{locop}),
there exists $\bm{\hat{\beta}}_S\in\partial\|\bm{\hat{\theta}}_S\|_1$
such that the equation (\ref{req1}) holds.
So if it happens with a high probability that the equation (\ref{req2}) holds and $\|\bm{\hat{\beta}}_{S^C}\|_{\infty}<1$,
thus the event $\mathbf{\Gamma}=\{({\bm{\hat{\theta}}_S}^T,\mathbf{0}^T)^T \mbox{is the unique optimal solution}\ \bm{\hat{\theta}}\ \mbox{to problem} (\ref{objectarg})\}$ happens with a high probability. Furthermore, the event $\{\hat{S}\subseteq S\}$ also happens with a high probability.
We are going to show these in the the following part of this section.

From equations (\ref{req1}) and (\ref{req2}), we have that:
\begin{equation}
\bm{\hat{\beta}}_{S^C}=\mathbf{B}_{S^C}^T\mathbf{B}_S(\mathbf{B}_S^T\mathbf{B}_S)^{-1}\bm{\hat{\beta}}_S+
\mathbf{B}_{S^C}^T(\mathbf{I}-\mathbf{B}_S(\mathbf{B}_S^T\mathbf{B}_S)^{-1}\mathbf{B}_S^T)\frac{\bm{\varepsilon}}{\lambda},
\end{equation}
\begin{equation}
\bm{\delta}_S=\bm{\hat{\theta}}_S-\bm{\theta}_S^*=(\mathbf{B}_S^T\mathbf{B}_S)^{-1}(\mathbf{B}_S^T\bm{\varepsilon}-\lambda\bm{\hat{\beta}}_S).
\end{equation}
For any $j\in S^C$,
we have that
$$\hat{\beta}_j=\mathbf{B}_j^T\mathbf{B}_S(\mathbf{B}_S^T\mathbf{B}_S)^{-1}\bm{\hat{\beta}}_S+
\mathbf{B}_j^T(\mathbf{I}-\mathbf{B}_S(\mathbf{B}_S^T\mathbf{B}_S)^{-1}\mathbf{B}_S^T)\frac{\bm{\varepsilon}}{\lambda}=\mu_j+\omega_j.$$

We assume that there exists an incoherence parameter $\gamma\in(0,1]$ such that
$\|\mathbf{B}_{S^C}^T\mathbf{B}_S(\mathbf{B}_S^T\mathbf{B}_S)^{-1}\|_{\infty}\leq 1-\gamma$,
where matrix norm $\|M\|_{\infty}=\max\limits_{i}\sum\limits_{j}|M_{ij}|$.
It is easy to obtain $|\mu_j|=|\mathbf{B}_j^T\mathbf{B}_S(\mathbf{B}_S^T\mathbf{B}_S)^{-1}\bm{\hat{\beta}}_S|\leq \|\mathbf{B}_{S^C}^T\mathbf{B}_S(\mathbf{B}_S^T\mathbf{B}_S)^{-1}\|_{\infty}\leq 1-\gamma$.
And let us consider $\omega_j=\mathbf{B}_j^T(\mathbf{I}-\mathbf{B}_S(\mathbf{B}_S^T\mathbf{B}_S)^{-1}\mathbf{B}_S^T)\frac{\bm{\varepsilon}}{\lambda}=
\frac{1}{\lambda}(c_1\varepsilon_1+\cdots+c_n\varepsilon_n)$, where $(c_1,\cdots,c_n)=\mathbf{B}_j^T(\mathbf{I}-\mathbf{B}_S(\mathbf{B}_S^T\mathbf{B}_S)^{-1}\mathbf{B}_S^T)$.
Thus $\omega_j$ is a subgaussian distribution with zero mean and parameter $\frac{\sigma^2}{\lambda}(c_1^2+\cdots+c_n^2)=
\frac{\sigma^2}{\lambda^2}\mathbf{B}_j^T(\mathbf{I}-\mathbf{B}_S(\mathbf{B}_S^T\mathbf{B}_S)^{-1}\mathbf{B}_S^T)\mathbf{B}_j$.
Since $\mathbf{B}_j^T\mathbf{B}_j=1$, this parameter is no more than $\frac{\sigma^2}{\lambda^2}$.
So $Pr(\max\limits_{j\in S^C}|\omega_j|\geq t)\leq 2(R-k)\exp(-\frac{\lambda^2t^2}{2\sigma^2})$, where $k$ is the cardinality of $S$.
By choosing $t=\frac{1}{2}\gamma$, we have that $Pr(\max\limits_{j\in S^C}|\omega_j|\geq \frac{1}{2}\gamma)\leq 2(R-k)\exp(-\frac{\lambda^2\gamma^2}{8\sigma^2})$.
Thus we have that
\begin{equation}\label{inequ1}
Pr(\max_{j\in S^C}|\hat{\beta}_j|>1-\frac{\gamma}{2})\leq Pr(\max\limits_{j\in S^C}|\omega_j|\geq \frac{1}{2}\gamma)\leq 2(R-k)\exp(-\frac{\lambda^2\gamma^2}{8\sigma^2}).
\end{equation}

Now let us consider about the upper bound of $\bm{\delta}_S$: $\|\bm{\delta}_S\|_{\infty}\leq \|(\mathbf{B}_S^T\mathbf{B}_S)^{-1}\mathbf{B}_S^T\bm{\varepsilon}\|_{\infty}+
\lambda\|(\mathbf{B}_S^T\mathbf{B}_S)^{-1}\|_{\infty}$.
Since $\lambda\|(\mathbf{B}_S^T\mathbf{B}_S)^{-1}\|_{\infty}$ has a fixed value, we only need to consider the first term.
For any $i\in S$, we have that $v_i=\mathbf{e}_i^T(\mathbf{B}_S^T\mathbf{B}_S)^{-1}\mathbf{B}_S^T\bm{\varepsilon}=c_1\varepsilon_1+\cdots+c_n\varepsilon_n$,
where $(c_1,\cdots,c_n)=\mathbf{e}_i^T(\mathbf{B}_S^T\mathbf{B}_S)^{-1}\mathbf{B}_S^T$.
If we assume that $\lambda_{min}(\mathbf{B}_S^T\mathbf{B}_S)\geq\mu$,
thus $v_i$ is a subgaussian distribution with
zero mean and parameter $\frac{\sigma^2}{\lambda}(c_1^2+\cdots+c_n^2)=\sigma^2\mathbf{e}_i^T(\mathbf{B}_S^T\mathbf{B}_S)^{-1}\mathbf{e}_i\leq\frac{\sigma^2}{\mu}$.
Thus $Pr(\max\limits_{i\in S}|v_i|>t)\leq 2k\exp(-\frac{t^2\mu}{2\sigma^2})$.
By choosing $t=\frac{\lambda}{2\sqrt{\mu}}$,
we have that \\
\begin{equation}\label{inequ2}
Pr(\max\limits_{i\in S}|v_i|>\frac{\lambda}{2\sqrt{\mu}})\leq 2k\exp(-\frac{\lambda^2}{8\sigma^2})\leq 2k\exp(-\frac{\lambda^2\gamma^2}{8\sigma^2}).
\end{equation}

By combining (\ref{inequ1}) and (\ref{inequ2}),
we have the probability inequality $Pr(\{\max\limits_{j\in S^C}|\hat{\beta}_j|>1-\frac{\gamma}{2}\}\bigcup\{\max\limits_{i\in S}|v_i|>\frac{\lambda}{2\sqrt{\mu}}\})\leq 2R\exp(-\frac{\lambda^2\gamma^2}{8\sigma^2})$.
Thus the probability inequality on the complementary set is that
$$Pr(\{\max\limits_{j\in S^C}|\hat{\beta}_j|\leq 1-\frac{\gamma}{2}\}\bigcap\{\max\limits_{i\in S}|v_i|\leq\frac{\lambda}{2\sqrt{\mu}}\})\geq 1-2R\exp(-\frac{\lambda^2\gamma^2}{8\sigma^2}).$$
Furthermore, we have that
\begin{equation}\label{orgconsist}
Pr(\mathbf{\Gamma}\bigcap\{\|\bm{\delta}_S\|_{\infty}\leq\frac{\lambda}{2\sqrt{\mu}}+\lambda\|(\mathbf{B}_S^T\mathbf{B}_S)^{-1}\|_{\infty}\})
\geq 1-2R\exp(-\frac{\lambda^2\gamma^2}{8\sigma^2}),
\end{equation}
where $\mathbf{\Gamma}=\{({\bm{\hat{\theta}}_S}^T,\mathbf{0}^T)^T \mbox{is the unique optimal}$
$\mbox{solution}\ \bm{\hat{\theta}}\ \mbox{to problem}\ (\ref{objectarg})\}$.

From the above discussion, we obtain the following Theorem \ref{probcon}, which illustrates the probabilistic consistency of
the optimal solution $\bm{\hat{\theta}}$ to problem (\ref{objectarg}).

\begin{theorem}\label{probcon}
Suppose that the sparse structure (\ref{sparsity}) exists,
the sparse signal $\bm{\theta}^*=({\bm{\theta}_S^*}^T,\mathbf{0}^T)^T$ and $\mathbf{B}$ has full column rank.
If there exist some parameters $\gamma$ and $\mu$ where  $0<\gamma<1$ and  $\mu>0$ such that
$\|\mathbf{B}_{S^C}^T\mathbf{B}_S(\mathbf{B}_S^T\mathbf{B}_S)^{-1}\|_{\infty}\leq 1-\gamma$
and $\lambda_{min}(\mathbf{B}_S^T\mathbf{B}_S)\geq\mu$,
we have that
\begin{equation}\label{consist}
Pr(\{\hat{S}\subseteq S\}\bigcap\{\|\bm{\delta}_S\|_{\infty}\leq\frac{\lambda}{2\sqrt{\mu}}+\lambda\|(\mathbf{B}_S^T\mathbf{B}_S)^{-1}\|_{\infty}\})
\geq 1-2R\exp(-\frac{\lambda^2\gamma^2}{8\sigma^2}),
\end{equation}
where $\hat{S}$ is the index set of non-zero entries in $\bm{\hat{\theta}}$, and
$\bm{\delta}_S=\bm{\hat{\theta}}_S-\bm{\theta}_S^*$ and
$\bm{\hat{\theta}}_S$ is the optimal solution of (\ref{locop}).
Furthermore, if the lower bound of the absolute values of elements in $\bm{\theta}_S^*$ is larger than $\lambda(\frac{1}{2\sqrt{\mu}}+\|(\mathbf{B}_S^T\mathbf{B}_S)^{-1}\|_{\infty})$,
we have that
\begin{equation}\label{consist2}
Pr(\{\hat{S}=S\})\geq 1-2R\exp(-\frac{\lambda^2\gamma^2}{8\sigma^2}).
\end{equation}
\end{theorem}
\begin{proof}
In terms of (\ref{orgconsist}), the first inequality (\ref{consist}) follows from $\{\hat{S}\subseteq S\}\supseteq\mathbf{\Gamma}$,
where $\mathbf{\Gamma}=\{({\bm{\hat{\theta}}_S}^T,\mathbf{0}^T)^T \mbox{is the unique optimal solution}\ \bm{\hat{\theta}}
\ \mbox{to problem}\ (\ref{objectarg})\}$.

If $\|\bm{\hat{\theta}}_S-\bm{\theta}_S^*\|_{\infty}=\|\bm{\delta}_S\|_{\infty}\leq
\frac{\lambda}{2\sqrt{\mu}}+\lambda\|(\mathbf{B}_S^T\mathbf{B}_S)^{-1}\|_{\infty}$
and the lower bound of the absolute values of elements in $\bm{\theta}_S^*$ is larger than $\frac{\lambda}{2\sqrt{\mu}}+\lambda\|(\mathbf{B}_S^T\mathbf{B}_S)^{-1}\|_{\infty}$,
it can be checked that the entries in $\bm{\hat{\theta}}_S$
and $\bm{\theta}_S^*$ of the same index have the same sign.
From (\ref{orgconsist}),
we can obtain the second inequality (\ref{consist2}).
\end{proof}

Theorem \ref{probcon} tells us that if we want to recover the sparsity in (\ref{sparsity}) with a probability $p$,
we should choose a $\lambda$ such that $1-2R\exp(-\frac{\lambda^2\gamma^2}{8\sigma^2})>p$
when we know the intrinsic parameters $\gamma$ and $\sigma^2$.
So to adaptively give a regularization parameter $\lambda$ based on the data $\mathcal{A}$,
we need to give two guesses on the intrinsic parameters $\gamma$ and $\sigma^2$.
We set $\lambda$ to zero in the Algorithm \ref{alg:Framwork}, and compute a estimated tensor $\mathcal{\hat{B}}=[\bm{\hat{\alpha}};\mathbf{\hat{X}},\mathbf{\hat{Y}},\mathbf{\hat{Z}}]_R$ from the tensor data $\mathcal{A}$.
The parameter $\sigma^2$ is estimated by using the variance $\hat{\sigma}^2$ of all the entries in the difference $\mathcal{A}-\mathcal{\hat{B}}$,
and the parameter $\gamma$ is set as $\hat{\gamma}=1-\max\{|\langle\mathbf{B}_i,\mathbf{B}_j\rangle||i\neq j\}$,
where $\mathbf{B}_i$ is the $i$-th column in $\mathbf{B}=(\mathbf{\hat{X}}\odot\mathbf{\hat{Y}})\odot\mathbf{\hat{Z}}$.
With regularization parameter $\hat{\lambda}=\frac{2}{\hat{\gamma}}\sqrt{2\hat{\sigma}^2\log(200R)}$,
the result of our algorithm is shown by using the simulated and real data in the next Section.

\section{Numerical Experiment}
In this section, we have four types of numerical experiments for testing the performance of our algorithm.
The codes of the first three experiments are written in Matlab with simulated data.
In all the simulations, the initial guesses are randomly generated. The stopping criterion used in the all experiments depends on two parameters:
one is the upper bound of the number of iteration iteration number (eg. iter\_max$=10000$),
and the other is a tolerance to decide whether convergence has been reached (eg. conv\_tol$=e^{-10}$).
The fourth numerical experiment is executed in C++ with OpenCV for surveillance video data. These experiments ran on a laptop computer with Intel i5 CPU 3.3GHz and 8G memory.

\subsection{Estimated rank}
We randomly create a tensor $\mathcal{A}\in\mathbb{R}^{10\times 10\times 10}$ with $5$ rank-one components,
and then use LRAT to estimate the rank of $\mathcal{A}$
along with the increment of the regularization parameter.
The upper bound $R$ of $\mbox{rank}(\mathcal{A})$ is fixed to $10$ in the algorithm
while the regularization parameter $\lambda$ varies from $0$ to $0.1$ by step $0.001$.
As shown in Figure \ref{fig.ex1}, the estimated rank $\hat{R}$ has a decreasing trend
as the parameter $\lambda$ increases for these particular random tensor examples.
Heuristically, the reason for this trend lies is in the minimization the objective function in (\ref{problem4}),
an increase in $\lambda$ reduces the value of $\|\bm{\hat{\alpha}}\|_1$ and thus the estimated rank $\hat{R}$.

\begin{figure}
\begin{center}
\begin{overpic}[scale=.5,tics=10]%
{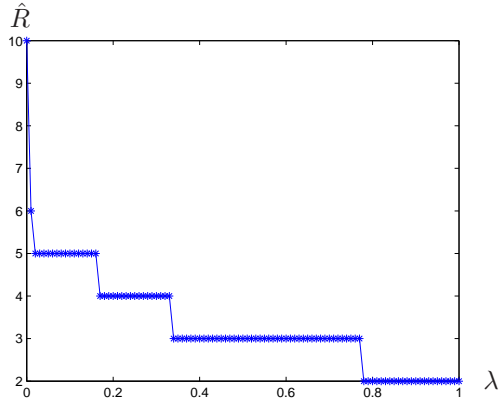}
\put(95,7){$\lambda$}
\put(10,72){$\hat{R}$}
\end{overpic}
\end{center}
\caption{Trend of the estimated rank $\hat{R}$.}\label{fig.ex1}
\end{figure}

\subsection{Accuracy of the estimated rank}
We randomly generate three kinds of tensors with various dimensions and
various rank-one component numbers (cn).
The estimated rank $\hat{R}$ is calculated with the regularization parameter
$\hat{\lambda}=\frac{2}{\hat{\gamma}}\sqrt{2\hat{\sigma}^2\log(200R)}$,
where $\hat{\sigma}^2$ and $\hat{\gamma}$ are computed as discussed in Section 5.
Table \ref{tab:1} shows the mean and standard deviation of the estimated rank.

\begin{table}[!htbp]
\caption{Mean and standard deviation of the estimated rank $\hat{R}$.}\label{tab:1}
\begin{center}
\begin{tabular}{|c|c|c|c|c|c|c|c|}
\hline
                    & $I=J=K=5$    & $I=J=K=10$      & $I=J=K=20$      \\ \hline
        $cn=2$        & 2.28 (0.87) & 3.25 (1.31) & 5.41 (1.85) \\\hline
        $cn=3$          & 3.15 (0.93) & 4.49 (1.12)& 7.2 (2.06)  \\\hline
        $cn=4$           & 3.6 (0.92)& 5.18 (1.16) & 8.35 (1.82) \\\hline
        $cn=5$        &*&   5.77 (1.29)& 9.98 (1.60)                 \\\hline
        $cn=8$        &*&   7.52 (1.01)&  10.88 (1.51)               \\\hline
        $cn=10$       &*&     *& 11.69 (1.50)             \\\hline
        $cn=15$       &*&     *&   14.11 (1.43)           \\\hline
\end{tabular}
\end{center}
\end{table}

For each component number $cn=2,3,4$,
we randomly generate $100$ tensors in $\mathbb{R}^{5\times 5\times 5}$
with $I=J=K=5$,
and then use the LRAT with the upper bound $R=5$ to compute the estimated rank $\hat{R}$.
As shown in Table \ref{tab:1},
when the rank-one component number $cn=3$,
the average estimation difference of $\hat{R}-cn$ is $0.15$ and the standard deviation of $\hat{R}$ is $0.93$.

Similarly, for each component number $cn=2,3,4,5,8$,
we randomly generate $100$ tensors in $\mathbb{R}^{10\times 10\times 10}$
and for $cn=2,3,4,5,8,10,15$, we randomly generate $100$ tensors in $\mathbb{R}^{20\times 20\times 20}$.
The upper bound $R$ is set to $I=10,20$.
The mean and standard deviation of $\hat{R}$ are shown in the last two columns of Table \ref{tab:1}.

\subsection{Comparison between LRAT and modALS}
In this subsection, we show the comparison between LRAT and modALS \cite{xu} on a toy model.
A tensor $\mathcal{A}$ in $\mathbb{R}^{5\times 5\times 5}$ is randomly generated with 3 rank-one components.
Based on the residual function, the modALS algorithm approximates tensor $\mathcal{A}$ by a tensor of five rank-one components,
while the LRAT algorithm solves (\ref{problem4}) and obtain an estimate on the rank of tensor $\mathcal{A}$. 

\begin{figure*}[!htbp]
\centering
\subfigure[Residual of LRAT and modALS] {\includegraphics[height=1.8in,width=2.5in]{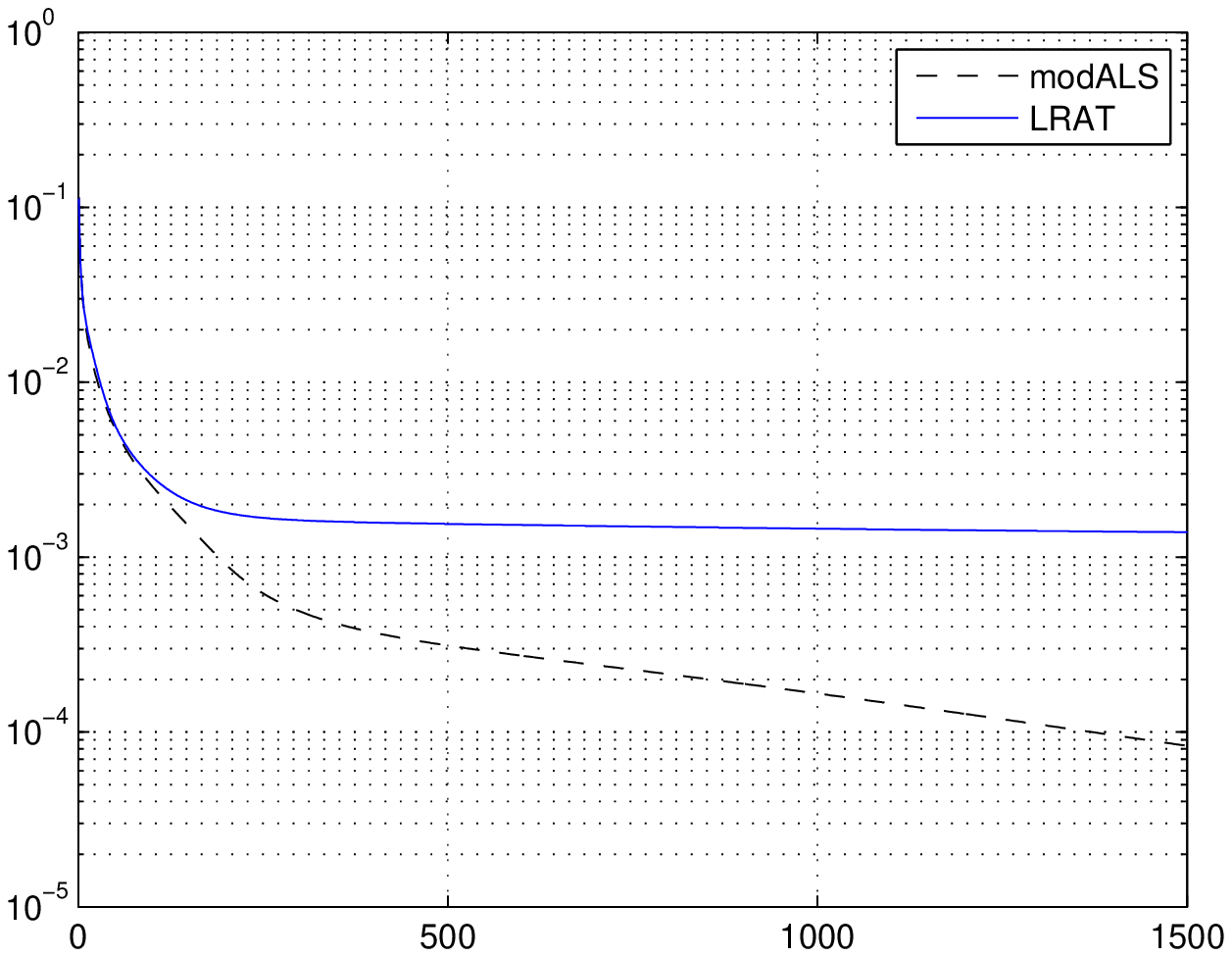}}
\subfigure[Objective function of LRAT] {\includegraphics[height=1.8in,width=2.5in]{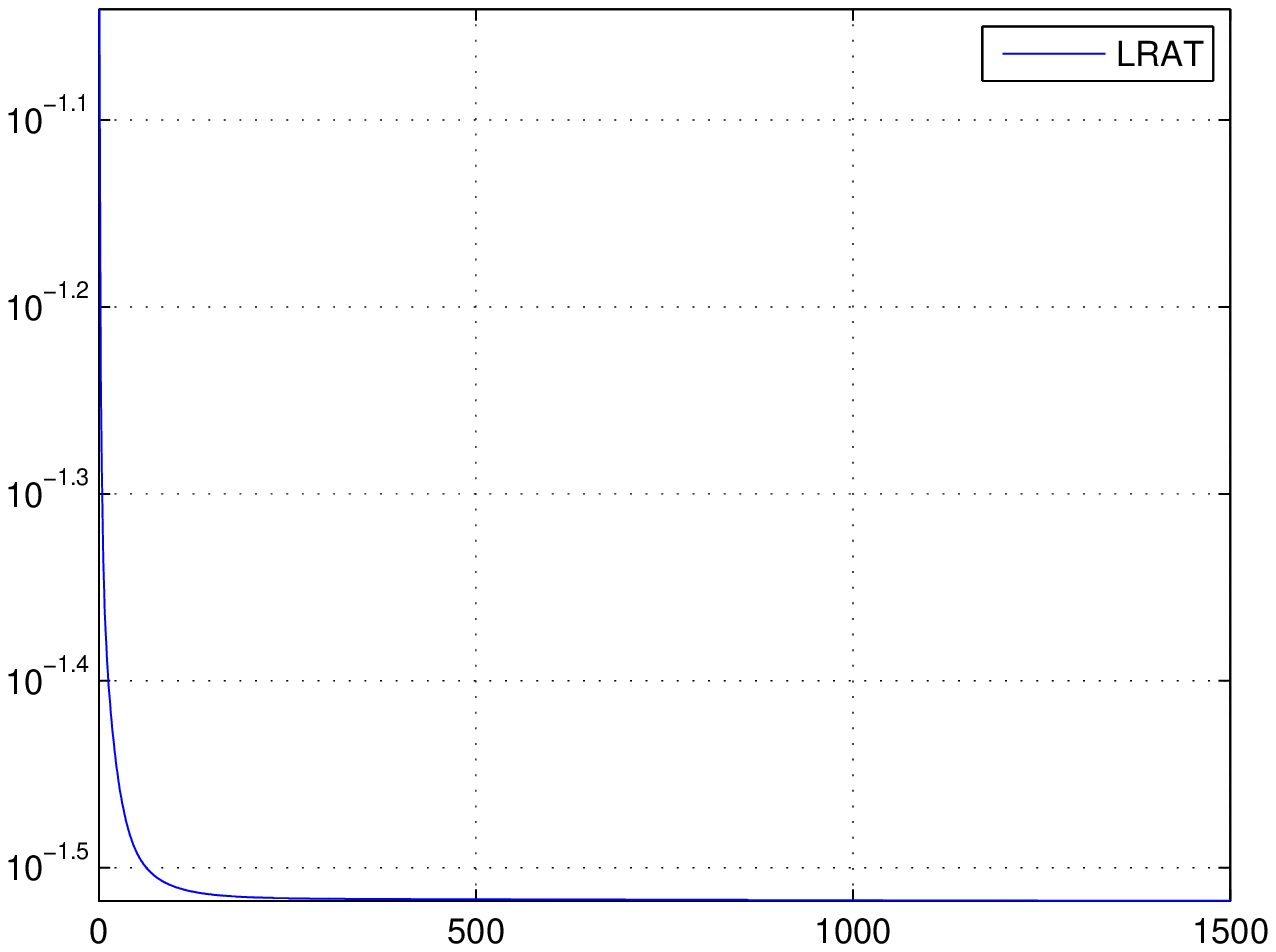}}
\caption{Comparison between LRAT and modALS.}\label{fig:comparison}
\end{figure*}

Figure \ref{fig:comparison} (a) demonstrates the residual function
$\|\mathcal{A}-[\hat{\bm{\alpha}};\hat{\mathbf{X}},\hat{\mathbf{Y}},\hat{\mathbf{Z}}]_R\|_F^2$
for the modALS and LRAT algorithms.
Compared to the LRAT, the modALS executed by using five rank-one components decreases monotonically and has a low misfit on the residual function.
Figure \ref{fig:comparison} (b) shows the objective function
$\frac{1}{2}\|\mathcal{A}-[\hat{\bm{\alpha}};\hat{\mathbf{X}},\hat{\mathbf{Y}},\hat{\mathbf{Z}}]_R\|_F^2+\hat{\lambda}\|\hat{\bm{\alpha}}\|_1$ in (\ref{problem4})
for the LRAT algorithm.
It decreases monotonically and provides an estimate on the number of rank-one components as shown in Section 6.2,
though the LRAT gave one less digit accuracy in the residual.

\subsection{Application in surveillance video}
Grayscale video data is a natural candidate for third-order tensors.
Due to the correlation between subsequent frames of video, there exists some potential low-rank mechanism in the data.
In this subsection, we apply the LRAT and the modALS to two surveillance videos\footnote{The original data is from http://perception.i2r.a-star.edu.sg/bk\_model/bk\_index.html} on Fountain and Lobby.
For each video of $220$ consecutive frames, we choose a region of interest with a resolution $30\times 30$.

\begin{figure*}[htb]
\centering
\subfigure[20 frames for Fountain] {\includegraphics[height=1.28in,width=1.6in]{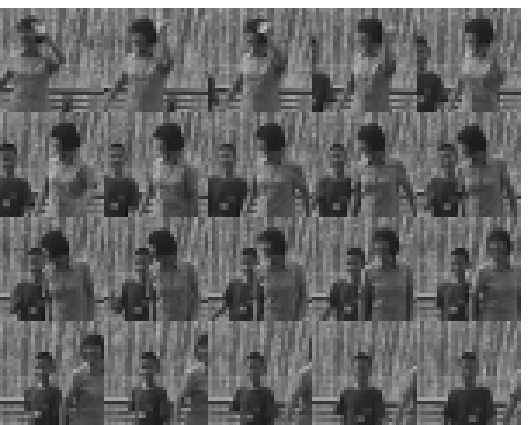}}
\subfigure[Results from LRAT] {\includegraphics[height=1.28in,width=1.6in]{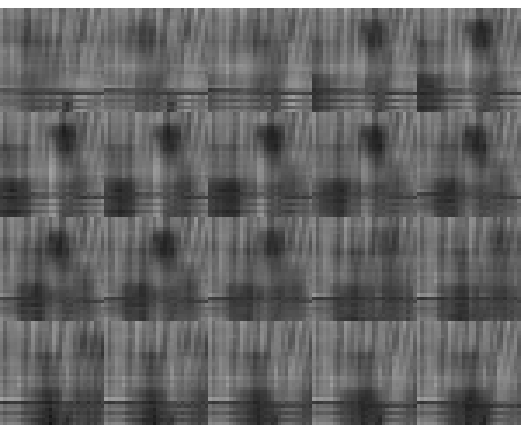}}
\subfigure[Results from modALS] {\includegraphics[height=1.28in,width=1.6in]{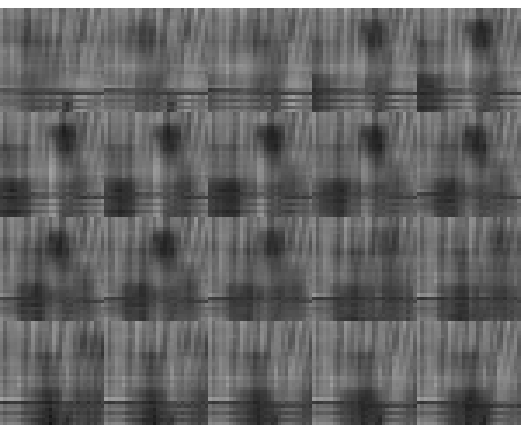}}
\subfigure[20 frames for Lobby] {\includegraphics[height=1.28in,width=1.6in]{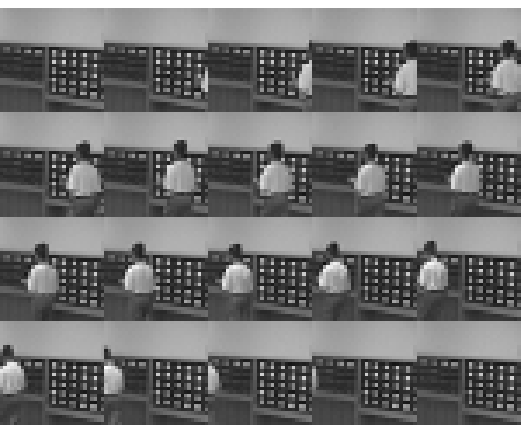}}
\subfigure[Results from LRAT] {\includegraphics[height=1.28in,width=1.6in]{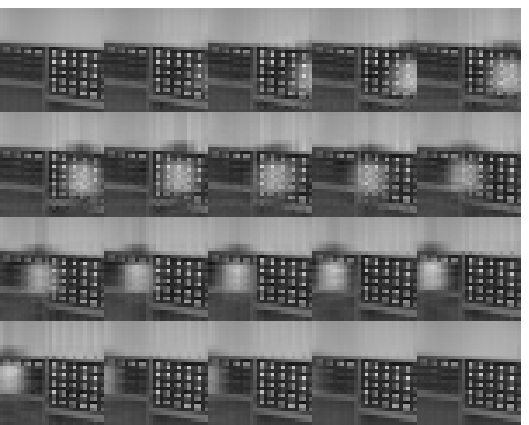}}
\subfigure[Results from modALS] {\includegraphics[height=1.28in,width=1.6in]{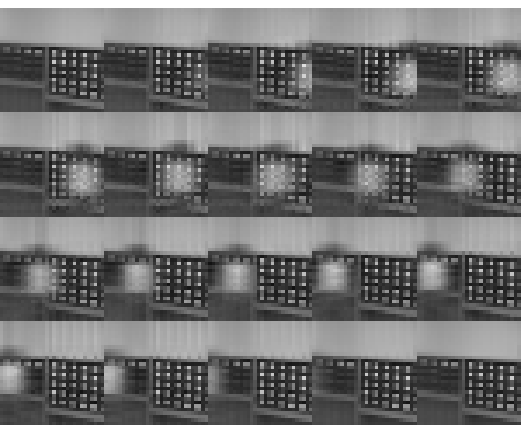}}
\caption{Computation results based on LRAT and  modALS.}\label{fig:modALS}
\end{figure*}

Figure \ref{fig:modALS} demonstrates simulation results on the LRAT and the modALS.
Here the upper bound $R$ is fixed to $400$.
Figure \ref{fig:modALS} shows 20 frames in the original video data $\mathcal{A}$ and those
frames estimated by the LRAT and the modALS.
The modALS provides an approximation with three factor matrices of $(30+30+220)\times 400$ elements.
For the LRAT algorithm, the regularization parameter $\hat{\lambda}$ is set to $\frac{2}{\hat{\gamma}}\sqrt{2\hat{\sigma}^2\log(200R)}$,
where $\hat{\sigma}^2$ and $\hat{\gamma}$ are computed as discussed in Section 5.
The estimated number of rank-one components in $\hat{\mathcal{B}}$ is $378$ for the Fountain video.
The representation of $\hat{\mathcal{B}}$ with three factor matrices only needs $(30+30+220)\times 378$ elements.
The estimated number of rank-one components is $392$ for the Lobby video,
and the representation with three factor matrices needs $(30+30+220)\times 392$ elements.
Compared to the modALS algorithm, the LRAT has a smaller estimated rank but it sacrifices more cpu time,
because the LRAT algorithm requires an instructive (a starter) choice on $\hat{\lambda}$.
For this case, we used the modALS algorithm to obtain a starter choice $\hat{\lambda}$ for LRAT.


\section{Conclusion and future work}
We propose an algorithm based on the proximal alternating minimization to detect the rank of tensors.
This algorithm comes from the understanding of the low-rank approximation of
tensors from sparse optimization.
We also provide some theoretical guarantees on the convergence of this algorithm
and a probabilistic consistency of the approximation result.
Moreover, we suggest a way to choose a regularization parameter for practical computation.
The simulation studies suggested that our algorithm can be used to detect the number of rank-one components in tensors.

The work presented in this paper have potential applications and extensions, especially in video processing
and latent component number estimation.
The ongoing work is to apply this low-rank approximation method
to moving object detection and video data compression.

\section{Acknowledgements}
This work is in part supported by the National Natural Foundation of China 11401092.

\end{document}